\documentclass[a4paper, 10pt]{amsart}
\usepackage[leqno]{amsmath}
\usepackage{amssymb,amsthm,MnSymbol, mathtools}
\usepackage[all]{xy}
\usepackage{mathrsfs}
\usepackage{tikz}
\usepackage{tikz-cd}
\usepackage{enumitem}
\usepackage[divide={2.5cm,*,2.5cm}]{geometry}
\usepackage{xcolor}
\usepackage{upgreek}
\usepackage{hyperref}
\usepackage{cleveref}

\numberwithin{equation}{section}

\theoremstyle{plain}
\newtheorem{thm}{Theorem}[section]
\newtheorem{thmx}{Theorem}
\newtheorem{prop}[thm]{Proposition}
\newtheorem{lem}[thm]{Lemma}
\newtheorem{cor}[thm]{Corollary}

\theoremstyle{definition}
\newtheorem{defn}[thm]{Definition}

\newtheorem{rem}[thm]{Remark}
\newtheorem{ex}[thm]{Example}

\newcommand{\A}{{\mathcal{A}}}
\newcommand{\Af}{{\mathbb{A}}}
\newcommand{\B}{{\mathcal{B}}}
\newcommand{\C}{{\mathcal{C}}}

\newcommand{\D}{{\mathcal{D}}}
\newcommand{\E}{{\mathcal{E}}}
\newcommand{\M}{{\mathcal{M}}}

\newcommand{\Pro}{{\mathbb{P}}}
\newcommand{\He}{{\mathcal{H}}}
\newcommand{\Hea}{{\mathcal{H}_\aff}}
\newcommand{\F}{{\mathbb{F}}}

\newcommand{\Ff}{{\mathfrak{F}}}
\newcommand{\G}{{\mathbf{G}}}
\newcommand{\Gm}{{\mathbb{G}_m}}
\newcommand{\GL}{{\mathrm{GL}}}
\newcommand{\PGL}{{\mathrm{PGL}}}
\newcommand{\SL}{{\mathrm{SL}}}
\newcommand{\R}{{\mathcal{R}}}

\newcommand{\LM}{{\mathcal{L}}}

\newcommand{\m}{{\mathfrak{m}_{\mathfrak{F}}}}
\newcommand{\OF}{{\mathcal{O}_{\mathfrak{F}}}}
\newcommand{\T}{{\mathcal{T}}}

\newcommand{\XX}{{\mathbb{X}}}
\newcommand{\Z}{{\mathbb{Z}}}
\newcommand{\Zq}{{\tilde{\mathcal{Z}}}}

\newcommand{\ac}{{\operatorname{ac}}}
\newcommand{\Ch}{{\operatorname{Ch}}}
\newcommand{\cind}{{\operatorname{c-ind}}}

\newcommand{\Coh}{{\operatorname{Coh}}}
\newcommand{\QCoh}{{\operatorname{QCoh}}}
\newcommand{\coker}{{\operatorname{coker}}}

\newcommand{\End}{{\operatorname{End}}}
\newcommand{\Ext}{{\operatorname{Ext}}}
\newcommand{\GP}{{\operatorname{GP}}}
\newcommand{\Gal}{{\operatorname{Gal}}}
\newcommand{\Ho}{{\operatorname{Ho}}}
\newcommand{\Hom}{{\operatorname{Hom}}}
\newcommand{\id}{{\operatorname{id}}}
\newcommand{\im}{{\operatorname{Im}}}
\newcommand{\ind}{{\operatorname{ind}}}
\newcommand{\Ind}{{\operatorname{Ind}}}
\newcommand{\Ker}{{\operatorname{Ker}}}
\newcommand{\Inj}{{\operatorname{Inj}}}

\newcommand{\Mod}{{\operatorname{Mod}}}

\newcommand{\pd}{{\operatorname{pd}}}
\newcommand{\pr}{{\operatorname{pr}}}
\newcommand{\Proj}{{\operatorname{Proj}}}
\newcommand{\Rep}{{\operatorname{Rep}}}
\newcommand{\reg}{{\operatorname{reg}}}
\newcommand{\Res}{{\operatorname{Res}}}
\newcommand{\Sing}{{\operatorname{Sing}}}

\newcommand{\Spec}{{\operatorname{Spec}}}
\newcommand{\Supp}{{\operatorname{Supp}}}
\newcommand{\triv}{{\operatorname{triv}}}

\newcommand{\aff}{{\operatorname{aff}}}

\title{Pro-$p$ Iwahori-Hecke modules in semisimple rank one and singularity categories}
\date{}
\author{Nicolas Dupr\'e}
\address{Bergische Universit\"at Wuppertal\\
Gau{\ss}stra{\ss}e 20\\
D--42119 Wuppertal, Germany}
\email{dupre@uni-wuppertal.de}

\begin{document}

\begin{abstract}
Let $\Ff$ be a non-archimedean local field of residue characteristic $p$ and $G$ be one of the groups $\GL_2(\Ff)$, $\SL_2(\Ff)$ or $\PGL_2(\Ff)$. Let $\He_G$ denote the pro-$p$ Iwahori-Hecke algebra of $G$ over $\overline{\F}_p$. We study the homotopy category $\Ho(\He_G)$ of Hovey's Gorenstein projective model structure on the category of $\He_G$-modules and relate it to the singularity category $\Sing(X_{q,\G})$ of an explicit scheme. When $G=\GL_2(\Ff)$, this scheme was first introduced by Dotto-Emerton-Gee \cite{DEG22}. We obtain in that case an equivalence $\Ho(\He_{\GL_2})\simeq \Sing(X_{q,\GL_2})$ and recover from this Gro{\ss}e-Kl{\"o}nne's mod-$p$ Langlands correspondence for Hecke modules \cite{GK20}, building on work of P{\'e}pin-Schmidt \cite{PeSch25_2}. We furthermore describe $\Ho(\He_G)$ completely explicitly when $G=\SL_2(\Ff)$ or $\PGL_2(\Ff)$, and make additional computations in the $\GL_2$ case.
\end{abstract}

\maketitle
\thispagestyle{empty}

\setcounter{tocdepth}{1}
\tableofcontents

\section{Introduction}

Let $p$ be a prime and $\Ff$ be a non-archimedean local field with residue field $\F_q$ of characteristic $p$. Let $\G$ be a split reductive group over the field $\Ff$ and let $G=\G(\Ff)$. Let $I$ be a pro-$p$ Iwahori subgroup of $G$ and let $\He_G=\overline{\F}_p[I\backslash G/I]$ be the corresponding \emph{pro-$p$ Iwahori-Hecke algebra} of $G$, i.e.\ the convolution algebra of $I$-biinvariant compactly supported functions on $G$. Write $\Mod(\He_G)$ for the category of left $\He_G$-modules.

The Hecke algebra $\He_G$ plays an important role in the mod-$p$ Langlands program. For instance, by work of Gro{\ss}e-Kl{\"o}nne \cite{GK20} it is known when $\G=\GL_n$ and $\Ff$ is an extension of $\mathbb{Q}_p$ that there is a functorial bijection $M\mapsto V(M)$ between the isomorphism classes of $n$-dimensional simple supersingular $\He_G$-modules and the isomorphism classes of irreducible, continuous representations $\Gal(\overline{\Ff}/\Ff)\to \GL_n(\overline{\F}_p)$. Furthermore, when $\G=\GL_2$ there is a scheme with connected components given by chains of projective lines, introduced by Dotto-Emerton-Gee \cite{DEG22} and P{\'e}pin-Schmidt \cite{PeSch25_2}, whose $\overline{\F}_p$-rational points are in canonical bijection with the isomorphism classes of semisimple 2-dimensional Galois representations. P{\'e}pin-Schmidt show that this scheme is a certain quotient of the centre $Z(\He_{\GL_2})$ of the Hecke algebra and use this relationship to reconstruct geometrically Gro{\ss}e-Kl{\"o}nne's bijection in that case. A more general link between $Z(\He_{\GL_n})$ and Galois representations was also obtained in the recent preprint \cite{Sch26}.

The algebra $\He_G$ is in general known to be a Gorenstein ring \cite{OS14} and it follows from work of Hovey \cite{Hov2} that $\Mod(\He_G)$ has a so-called Gorenstein projective model structure. In earlier work \cite{Dup25} we investigated the associated homotopy category $\Ho(\He_G)$, which is the localisation of $\Mod(\He_G)$ obtained after trivialising all modules of finite projective dimension. In particular, we showed that distinct simple supersingular $\He_G$-modules remain generically non-isomorphic in $\Ho(\He_G)$. The goal of the present work is to describe the category $\Ho(\He_G)$ in more details when $\G$ has semisimple rank one, and to relate it to Galois representations via the aforementioned constructions.

From now on, we assume that $\G=\mathrm{GL_2}$, $\mathrm{SL_2}$ or $\mathrm{PGL_2}$. The first part of our paper aims to explicitly describe the category $\Ho(\He_G)$, which we recall is by construction triangulated. Our first result achieves this when $\G$ is semisimple.

\begin{thmx}[{\Cref{description_PGL2_SL2}}]\label{thmA}
There is an equivalence
$$
\Ho(\He_{\SL_2})\simeq \C\times \prod_{\text{$\gamma$ regular}} \left(\Mod(k^2)_{2,\triv}\times \Mod(k^2)_{2,\triv}\right),
$$
of triangulated categories, where $\C=0$ if $p=2$ and $\C=\Mod(k)_\triv\times \Mod(k)_\triv$ if $p>2$. Assuming $p>2$, there is also an equivalence
$$
\Ho(\He_{\PGL_2})\simeq \prod_{\text{$\gamma$ regular}} \Mod(k^2)_{2,\triv}
$$
of triangulated categories.
\end{thmx}

The notation $\Mod(k^2)_{2,\triv}$ and $\Mod(k)_\triv$ in the above statement refers to certain rather simple triangulated structures on these categories, cf.\ \Cref{triangle_modk2}. The products in the statement of the above theorem are indexed by the so-called regular Weyl group orbits in the character group of the finite torus $\mathbf{T}(\F_q)$. The reason for the extra factor $\C$ when $p>2$ and $\G=\SL_2$ is that, in that case, the algebra $\He_{\SL_2}$ has an additional component corresponding to the `sign' character $\sigma$ of $\mathbf{T}(\F_q)$ which sends a generator to $-1$. The proof of \Cref{thmA} uses Bruhat-Tits theory and the canonical resolutions of Ollivier-Schneider \cite{OS14} to reduce the question to describing the homotopy category $\Ho(\He_{x_0})$ of the finite Hecke algebra associated to the standard maximal compact subgroup of $G$, cf.\ \Cref{Delta_equiv}. The latter homotopy category can then be computed directly, by describing the module categories rather explicitly (cf.\ \Cref{classification_all} and \Cref{htpy_cat_Hx}). We note that this reduction argument to $\Ho(\He_{x_0})$ has an analogue on the side of smooth $G$-representations, cf.\ \Cref{Delta_equiv_reps}.

We then aim to study $\Ho(\He_G)$ in terms of singularity categories. We recall that the singularity category of a Noetherian separated scheme $X$, as defined by Krause \cite{Kra05}, is the triangulated category of chain complexes of injective quasi coherent sheaves on $X$ up to chain homotopy. These categories also feature in the geometric Langlands program under the more general definition $\Ind (\Coh(X))/\QCoh(X)$, see \Cref{geom_Lang} and the references therein. Our main result relates $\Ho(\He_G)$ to the singularity category $\Sing(X_{q,\G})$ of a projective scheme. When $\G\neq\SL_2$, the scheme $X_{q,\G}$ is constructed very explicitly as a quotient of $\Spec(Z(\He_G))$ and parametrises suitable Galois representations. In particular, $X_{q,\GL_2}$ is the aforementioned scheme of Dotto-Emerton-Gee and P{\'e}pin-Schmidt. The case $\G=\SL_2$ is slightly different because $\He_{\SL_2}$ has in a suitable sense more singularities than its centre. We will elaborate further on this below. We construct $X_{q,\SL_2}$ instead as a quotient of $\Spec(\Zq)$, for a certain explicit central subalgebra $\Zq$ of the affine Hecke algebra $\He_{\GL_2,\aff}$ which contains $Z(\He_{\SL_2})$.

\begin{thmx}[{\Cref{Langlands_singularities}, \Cref{Langlands_PGL2} \& \Cref{Langlands_SL2}}]\label{thmB}
Let $\G$ be as above and assume that $p>2$. There is a triangulated functor $\mathscr{L}_{\G,*}:\Ho(\He_{\G})\to \Sing(X_{q,\G})$, which is an equivalence when $\G=\GL_2$ or $\PGL_2$. Furthermore, there is a well-defined induced map
\[
\left\{\substack{\text{simple supersingular} \\ \text{$\He_G$-modules of} \\ \text{infinite projective dimension}}\right\}_{/\cong} \longrightarrow \left\{ \substack{\text{singular $\overline{\F}_p$-points}\\ \text{of $X_{q,\G}$}}\right\}
\]
given by $M\mapsto \Supp(\mathscr{L}_{\G,*}(M))$. This map satisfies the following:
\begin{enumerate}
    \item it is a bijection if $\G=\GL_2$ or $\PGL_2$ and agrees with Gro{\ss}e-Kl{\"o}nne's map, in other words $\Supp(\mathscr{L}_{\G,*}(M))=V(M)$, when $\Ff$ is an extension of $\mathbb{Q}_p$; and
    \item it is surjective if $\G=\SL_2$, with fibers equal to the $L$-packets of simple supersingular modules as defined in \cite{Koz2}.
\end{enumerate}
\end{thmx}

Part (i) above makes sense because of the canonical bijection between $X_{q,\GL_2}(\overline{\F}_p)$ and the set of isomorphism classes of semsimiple 2-dimensional Galois representations from \cite{PeSch25_2} and because $X_{q,\PGL_2}$ is by definition a closed subscheme of $X_{q,\GL_2}$. The notion of support in the statement was defined by Krause in \cite{Kra05} and can be characterised in terms of injective dimensions in the affine case, see \Cref{support_lem}. By construction, it is here always a subset of the singular locus of $X_{q,\G}$. We also point out that the restriction to simple supersingulars of infinite projective dimension is rather mild by a result of Koziol \cite{Koz}, see \Cref{recollections_Hecke} for the details.

The functor $\mathscr{L}_{\G,*}$ is constructed by utilising a certain \emph{spherical} Hecke module, which was introduced in \cite{PeSch23} for the generic pro-$p$ Iwahori-Hecke algebra of $\GL_2$. We define analogues of this construction for $\PGL_2$ and $\SL_2$ which fit our purposes. The proof of \Cref{thmB} then mainly consists in showing that pushing forward along the quotient morphism $\Spec(Z(\He_{\GL_2}))\to X_{q,\GL_2}$ from \cite{PeSch25_2} induces an equivalence of singularity categories. We also give an explicit description of $\mathscr{L}_{\SL_2,*}$ via \Cref{thmA}, cf.\ \Cref{Quillen_equiv_SL2}. Heuristically, the above says that our functor $\mathscr{L}_{\G,*}$ lifts the known mod-$p$ Langlands correspondences for Hecke modules to a nice functor between singularity categories. In particular, the fact that $\mathscr{L}_{\SL_2,*}$ is not an equivalence may be viewed as analogous to the mod-$p$ Langlands correspondence not being a bijection in that case.

We also note that Ardakov-Schneider constructed a scheme $\tilde{X}_{q,\SL_2}$ in \cite{ArdSch24}, given by a disjoint union of two chains of projective lines, using the action of $Z(\He_{\SL_2})$ on the Hecke-Ext algebra. Our scheme $X_{q,\SL_2}$ is obtained by gluing one more $\Pro^1$ at the end of one of their two chains, and that end point corresponds exactly to the aforementioned $\sigma$-component of $\He_{\SL_2}$. This component has a non-zero contribution to $\Ho(\He_{\SL_2})$ as seen in \Cref{thmA}, but its centre is a polynomial algebra and thus has no singularities. Our result says that, morally, that end point of $\tilde{X}_{q,\SL_2}$ should be a singularity in order to fully capture $\Ho(\He_{\SL_2})$. We note that an interpretation of the $\overline{\F}_p$-rational points of $\tilde{X}_{q,\SL_2}$ in terms of Galois representations is expected to appear in \cite{Yin26}, and that the sheaves in $\Sing(X_{q,\SL_2})$ are supported on $\tilde{X}_{q,\SL_2}$.

At the end of our paper, we do some additional computations in the case $\G=\GL_2$. Namely, we show that $\Ho(\He_{\GL_2})$ is equivalent to the derived category of an explicit DGA by singling out a generator, cf.\ \Cref{generator_Ho} and \Cref{computation_DGA}. This generator is closely related to the aforementioned spherical module. Finally, we show that the endomorphism rings of the simple supersingular modules in $\Ho(\He_{\GL_2})$ are isomorphic to the regular component of the finite Hecke algebra $\He_{x_0}$, cf.\ \Cref{endom_comp}.

\subsection*{Acknowledgments} The genesis of this work came from a question of Vytautas Pa{\v s}k{\=u}nas on the relationship between $\Ho(\He_{\GL_2})$ and the constructions of Dotto-Emerton-Gee. We thank him for his insightful question. We also thank Jan Kohlhaase and Ziqian Yin for useful conversations and for their comments on this work, and Martin Kalck for pointing out useful references.

\subsection*{Conventions and notation} We fix a prime $p$ and a nonarchimedean local field $\Ff$ of residue characteristic $p$. We write $\OF$ for the valuation ring of $\Ff$ with maximal ideal $\m$, uniformiser $\varpi$, and residue field $\F_q$. Throughout this paper we fix $k\coloneqq \overline{\F}_p$ and an embedding $\F_q\subseteq k$. For any group $H$, the trivial $H$-representation on $k$ will be denoted by $\mathbf{1}$.

If we denote an adjunction by $F:\C\rightleftarrows\D:U$ then $F$ is always assumed to be left adjoint to $U$. The unit (resp.\ the counit) of an adjunction will always be denoted by $\eta$ (resp.\ $\varepsilon$). We also say that $F$ {\it preserves} (resp.\ {\it reflects}) a property (P) if $F*$ (resp.\ $*$) has property (P) whenever $*$ (resp.\ $F*$) does. All rings considered will be assumed to be unital and associative. For such a ring $R$ we will denote by $\Mod(R)$ the category of left $R$-modules. When working with chain complexes of $R$-modules, our convention is that differentials have degree $+1$.

\section{Preliminaries}

\subsection{Assumptions and notation for $\G$} We let $\G$ be a split reductive group over $\Ff$ and set $G=\G(\Ff)$. We fix a maximal split $\Ff$-torus $\mathbf{T}$ of $\G$. The torus $T=\mathbf{T}(\Ff)$ of $G$ has a maximal compact subgroup $T^0$, and we let $T^1$ be its maximal pro-$p$ subgroup. The quotient $T^0/T^1$ is then a finite group which we denote by $T(\F_q)$.

We let $\mathscr{X}$ denote the semisimple Bruhat-Tits building. We fix a chamber $C$ in the standard apartment of $\mathscr{X}$ and a hyperspecial vertex $x_0$ in $\overline{C}$. Given a face $F$ in $\mathscr{X}$, we let $P_F$ denote the corresponding parahoric subgroup of $G$ and $P_F^\dagger$ be the stabiliser subgroup of $F$ in $G$. We now let $I$ be the pro-$p$ Iwahori subgroup corresponding to our choice of central chamber, i.e.\ the pro-$p$ radical of $P_C$.

Note that $T(\F_q)$ is a normal subgroup of $\widetilde{W}\coloneqq N_G(T)/T^1$ with quotient $W=N_G(T)/T^0$. The choice of $C$ determines a set of positive roots and hence a basis in the root system of $(\G, \mathbf{T})$. This in turn gives length function $\ell$ on the finite Weyl group $N_G(T)/T$, which extends to both $W$ and $\widetilde{W}$. There is an equality $W=\Omega\rtimes W_\aff$, where $\Omega=\{w\in W\,:\,\ell(w)=0\}$ and $W_\aff$ denotes the affine Weyl group generated by the set of simple affine reflections fixing the walls of $C$. The group $\Omega$ is finitely generated abelian and we write $\Omega_{\text{tor}}$ for its (finite) torsion subgroup. We will also denote by $\widetilde{\Omega}$ the pre-image of $\Omega$ in $\widetilde{W}$.

We will be primarily be interested in the case where $\G=\mathrm{SL_2}$, $\mathrm{PGL_2}$ or $\mathrm{GL_2}$. In that case, we will always choose $\mathbf{T}$ to be the standard torus of diagonal matrices, and moreover $C$ and the central vertex $x_0$ are fixed so that $I=\begin{psmallmatrix}
1+\m & \OF\\
\m & 1+\m
\end{psmallmatrix}$ and $P_{x_0}=\SL_2(\OF)$, $\PGL_2(\OF)$ or $\GL_2(\OF)$ as is usual. When $\G=\mathrm{PGL_2}$ or $\mathrm{GL_2}$, we write $\omega$ for (the image of) the matrix $\begin{psmallmatrix}
    0 & \varpi\\
    1 & 0
\end{psmallmatrix}$ in $G$. By abuse of notation, we will also denote by $\omega$ the image of that matrix in $W$ and $\widetilde{W}$. It then holds that $\Omega=\langle \omega\rangle$, so that $\Omega\cong \Z$ for $\G=\GL_2$ and $\Omega\cong \Z/2\Z$ for $\G=\PGL_2$. In the case $\G=\SL_2$ we have $\Omega=1$. We let $x_1=\omega\cdot x_0$ and note that $C$ is the chamber/edge joining $x_0$ and $x_1$. We also write $\alpha$ for the unique simple root in the root system, and $s_0\coloneqq s_\alpha$ and $s_1$ for the two simple affine reflections. The group $\Omega$ acts on the affine Weyl group by conjugation and we note that $\omega s_0\omega^{-1}=s_1$ and that $\omega^2$ acts trivially.

\subsection{Hecke algebras and supersingular modules}\label{recollections_Hecke} The pro-$p$ Iwahori Hecke algebra is the convolution algebra $\He_G=k[I\backslash G/I]$ of $I$-biinvariant compactly supported functions on $G$. We will sometimes omit the subscript $G$ when dealing with general $G$ or when the group is clear from context. The set $I\backslash G/I$ is in bijection with $\widetilde{W}$ and there is an associated Iwahori-Matsumoto basis $\{T_w\mid w\in \widetilde{W}\}$ of $\He$. For each face $F$ of $\mathscr{X}$, we have corresponding convolution algebras $\He_F^\dagger=k[I\backslash P_F^\dagger/I]$ and $\He_F=k[I\backslash P_F/I]$, and there are canonical inclusions $\He_F\subseteq \He_F^\dagger\subseteq \He$.

Specialising to our rank one situation, recall that the affine Hecke algebra $\Hea$ of $G$ is the $k$-algebra generated by $T_t$ ($t\in T(\F_q)$), $T_{s_0}$ and $T_{s_1}$, satisfying the braid relations and the quadratic relations
$$
T_{s_i}^2=T_{s_i}\cdot \left(|\mu_\alpha|\sum_{t\in\alpha^\vee(\F_q^\times)}T_{t}\right)
$$
for $i=0,1$, where $\alpha^\vee$ denotes the unique positive coroot of $\G$ and $\mu_\alpha\cong \ker(\alpha^\vee|_{\F_q^\times})$ has order one for $\G\neq \PGL_2$ or two otherwise, cf.\ \cite[\S 2.2, eqs.\ (27)-(28), \& \S 2.1.6, eq.\ (13)]{OS19}. From the quadratic relations, any character $\chi$ of $\Hea$ automatically satisfies $\chi(T_{s_i})\in\{0, -1\}$ for $i=0,1$. We say that $\chi$ is a \emph{supersingular character} in either one of the following two situations:
\begin{enumerate}
    \item the restriction $\chi|_{\alpha^\vee(\F_q^\times)}$ is non-trivial, in which case we automatically have $\chi(T_{s_i})=0$ for $i=0,1$ by the quadratic relations; or
    \item the restriction $\chi|_{\alpha^\vee(\F_q^\times)}$ is trivial and $\chi(T_{s_0})\neq \chi(T_{s_1})$.
\end{enumerate}
By Koziol's result \cite[Theorem 7.7]{Koz}, we have that a supersingular character $\chi$ has finite projective dimension if and only if it is given as in (ii) above.

We then have $\He_{\SL_2}=\He_{\SL_2,\aff}$. A supersingular simple $\He_{\SL_2}$-module will then just mean a supersingular character as above. In the other two cases, $\He$ is generated over $\Hea$ by $T_\omega^{\pm 1}$. This satisfies the relations $T_\omega T_{s_0} T_\omega^{-1}=T_{s_1}$ and $T_\omega T_t T_\omega^{-1}=T_{s_0\cdot t}$ for all $t\in T(\F_q)$. In particular, $T_\omega$ normalises $\Hea$. We also note that $\omega^2=\begin{psmallmatrix}
    \varpi & 0\\
    0 & \varpi
\end{psmallmatrix}$, and consequently we have that $T_\omega^2=1$, resp.\ $T_\omega^2$ is central, in $\He_{\PGL_2}$, resp.\ in $\He_{\GL_2}$. A simple \emph{supersingular} $\He$-module is in these two cases then a module $M$ defined as follows. We have
$$
M=ke_0\oplus ke_1
$$
with
$$
h\cdot e_0=\chi(h)e_0 \quad\text{and}\quad h\cdot e_1=\chi^\omega(h)e_1\coloneqq \chi(T_{\omega^{-1}} h T_{\omega})e_1
$$
for all $h\in \Hea$, where $\chi$ is a supersingular character of $\Hea$ as above. Furthermore, $T_\omega e_0=e_1$ and $T_\omega e_1=\lambda e_0$ for some fixed scalar $\lambda\in k^\times$. When $\G=\mathrm{PGL_2}$, we necessarily have $\lambda=1$.

\begin{rem}
Concretely, the relationships between the Hecke algebras of all three groups is as follows. We have that $\He_{\SL_2}$ is a subalgebra of $\He_{\GL_2, \aff}$ and $\He_{\GL_2, \aff}=k[T_{\GL_2}(\F_q)]\otimes_{k[T_{\SL_2}(\F_q)]}\He_{\SL_2}$ as $(k[T_{\GL_2}(\F_q)], \He_{\SL_2})$-bimodules. Meanwhile, we have
$$
\He_{\PGL_2}=\He_{\GL_2}/(T_\omega^2-1, T_t-1 : t\in \F_q^\times \cdot \id).
$$
so that $\He_{\PGL_2}$ is a quotient of $\He_{\GL_2}$.
\end{rem}

Additionally to the above, we will consider the subalgebras of $\He$ corresponding to the vertices $x_0$ and $x_1$ and the chamber $C$. Namely, $\He_{x_i}$ ($i=0,1$) is the subalgebra generated by $T(\F_q)$ and $T_{s_i}$ and $\He_C$ is the group algebra $k[T(\F_q)]$, viewed as a subalgebra of $\He$. Note that, being semisimple, $\He_C$ has global dimension 0.

If $\G=\SL_2$, we then have $\He_C^\dagger=\He_C$ and $\He_{x_i}^\dagger=\He_{x_i}$ ($i=0,1$). When $\G=\PGL_2$ or $\GL_2$, we instead have that $\He_C^\dagger=k[\widetilde{\Omega}]$ is the subalgebra of $\He$ generated by $\He_C$ and $T_\omega$, and $\He_{x_i}^\dagger$ ($i=0,1$) is the subalgebra of $\He$ generated by $\He_{x_i}$ and $T_\omega^2$. For $\G=\PGL_2$, $\He_C^\dagger$ is the group algebra $k[T(\F_q)\rtimes \Z/2\Z]$ while for $\G=\GL_2$ it is a skew-Laurent polynomial algebra over $\He_C$, and hence necessarily of global dimension 1 in the latter case. Finally, if $\G=\PGL_2$ then $\He_{x_i}^\dagger=\He_{x_i}$ ($i=0,1$), while for $\G=\GL_2$ we have that $\He_{x_i}^\dagger$ is isomorphic to the Laurent polynomial ring $\He_{x_i}[Z^{\pm1}]$.

\subsection{Gorenstein rings and homotopy categories} We recall that a ring $R$ is called \emph{Gorenstein} if it is left and right Noetherian and has finite selfinjective dimension both as a left and as a right module. In this case, the left and right selfinjective dimensions of $R$ have equal value $n$, say, and an $R$-module has finite projective dimension if and only if it has finite injective dimension, in which case both of these dimensions are bounded above by $n$, cf.\ \cite[9.1.8 \& 9.1.10]{EnJeBook1}. We will call such an $R$-module \emph{trivial}. The relevance of this notion to us is that the $k$-algebras $\He$, $\He_F^\dagger$ are all Gorenstein, cf.\ \cite[Theorem 0.1 \& Proposition 5.5]{OS14}, and furthermore the algebras $\He_F$ $(F\subseteq \overline{C})$ are even self-injective, cf.\ \cite[Example 4.1(ii)]{KD24}.

\subsubsection{Hovey's model structures on $\Mod(R)$} We recall that if $R$ is a Gorenstein ring then a module $M\in\Mod(R)$ is called \emph{Gorenstein projective}, resp.\ \emph{Gorenstein injective}, if there is an isomorphism $M\cong Z_0(E_\bullet)$, where $E_\bullet$ is an acyclic chain complex of $R$-modules which is termwise projective, respectively termwise injective. Such a chain complex will then be called a \emph{complete resolution} for $M$. Furthermore, we recall that a Gorenstein projective module, resp.\ Gorenstein injective module, is trivial if and only if it is projective, resp.\ injective, cf.\ \cite[Corollary 8.5]{Hov2}. We also recall as a special case that when $R$ is self-injective then the classes of injective and projective modules coincide, and hence every $R$-module has a complete resolution by combining an injective and a projective resolution. Thus every module is both Gorenstein projective and Gorenstein injective in that case.

\begin{ex}\label{princ_example}
We provide an example of a non-trivial Gorenstein projective module for the ring $A=k[X_1, X_2]/(X_1X_2)$, as this will play a role for us later. Note that this ring is Gorenstein by \cite[0AWX \& 0DWA]{stack} but is not self-injective. We consider the module $M\coloneqq A/X_1A\cong X_2A$. This occurs as the zero-th cycles of the complete resolution
\begin{equation*}
\cdots\xrightarrow{X_1\cdot}A\xrightarrow{X_2\cdot}A\xrightarrow{X_1\cdot}A\xrightarrow{X_2\cdot}\cdots
\end{equation*}
and thus is Gorenstein projective. Truncating this complex gives a projective resolution for $M$, and using this one computes that
$$
\Ext^j_A(M,M)=\begin{cases}
M & \text{if $j=0$}\\
M/X_2M\cong k &\text{if $j\geq 2$ is even}\\
0 & \text{if $j$ is odd}
\end{cases}.
$$
In particular, $M$ has infinite projective dimension. A completely analogous computation holds for the module $A/X_2A$, which is thus also Gorenstein projective of infinite projective dimension.
\end{ex}

Using these notions, Hovey constructed the so-called Gorenstein projective and Goreinstein injective model structures on $\Mod(R)$, cf.\ \cite[Theorem 8.6]{Hov2}. We do not recall here the formalism of (abelian) model categories, we refer the reader instead to our previous works \cite[\S 1]{KD24} and \cite[\S 2.4]{Dup25} as well as the references therein. In this paper we will exclusively work with the aforementioned model structures on $\Mod(R)$, which are defined as follows. In the Gorenstein projective model structure, the fibrations are the epimorphisms and the cofibrations are the monomorphisms with Gorenstein projective cokernel. Dually, the Gorenstein injective model structure has fibrations the epimorphisms with Gorenstein injective kernel and cofibrations the monomorphisms. In both cases, a fibration, resp.\ cofibration, is a weak equivalence if and only if its kernel, resp.\ cokernel, is trivial. Furthermore, the factorisation axioms and \cite[Lemma 5.8]{Hov2} imply that the weak equivalences in both of these model structures are the same, namely they are given by the smallest class of morphisms closed under composition and containing all epimorphisms with trivial kernels and all monomorphisms with trivial cokernel. Consequently, the two model structures have the same homotopy category, which we will denote by $\Ho(R)$ (see also \cite[Remark 2.12]{Dup25}).

If $M, N\in\Mod(R)$ are two (left) $R$-modules then we will set $[M,N]_R\coloneqq \Hom_{\Ho(R)}(M,N)$. By \cite[Proposition 9.1]{Hov2}, whenever $M$ is Gorenstein projective we have
\begin{equation}\label{Hom_Ho1}
[M,N]_R=\Hom_R(M,N)/\sim
\end{equation}
where the homotopy relation $\sim$ trivialises those maps which factor through a projective $R$-module. 
We next recall that $\Ho(R)$ is a triangulated category with arbitrary direct sums, where direct sums are computed by taking the image in $\Ho(R)$ of direct sums in $\Mod(R)$, cf.\ \cite[Theorem 6.34 \& Corollary 8.9(i)]{GilBook}. The triangulated structure is given as follows:
\begin{itemize}
    \item For $M\in \Mod(R)$, its shift $\Sigma(M)$ is given as the cokernel of any injection into a module of finite projective dimension. This is well-defined up to isomorphism in $\Ho(R)$.
    \item Given a short exact sequence
    $$
    0\to M_1\to M_2\to M_3\to 0
    $$
    in $\Mod(R)$, there is a canonical way of constructing a morphism $M_3\to \Sigma M_1$ in $\Ho(R)$ so that the resulting triangle $M_1\to M_2\to M_3\to \Sigma M_1$ is distinguished (see e.g.\ \cite[eq.\ (2.11)]{Dup25}, applied to the Gorenstein injevctive model structure). Furthermore, any distinguished triangle in $\Ho(R)$ is up to shift isomorphic to one of this form.
\end{itemize}
Finally, by \cite[Theorem 9.3]{Hov2} we have for a Gorenstein projective module $M$, any $N\in\Mod(R)$, and any $i\geq 1$ that
\begin{equation}\label{Hom_Ho3}
[\Sigma^{-i}M,N]_R\cong[M,\Sigma^i N]_R\cong \Ext_R^i(M,N).
\end{equation}

\subsubsection{Tensor-hom adjunctions between homotopy categories}\label{tensor_hom_section} Let $F:\C\rightleftarrows\D: U$ be an adjunction between the underlying categories of two model categories. We recall that this adjunction is called \emph{Quillen} when $F$ preserves cofibrations and trivial cofiabrations (or equivalently, if $U$ preserves fibrations and trivial fibrations). In that case the functor $F$, resp.\ $U$, is called \emph{left Quillen}, resp.\ \emph{right Quillen}. We will then denote by $\mathbf{L}F:\Ho(\C)\rightleftarrows\Ho(\D):\mathbf{R}U$ the induced total derived adjunction between the homotopy categories.

Suppose now that $R$ and $S$ are two Gorenstein rings, and let $X$ be an $(S,R)$-bimodule. Then we write $\LM_X$ and $\R_X$ for the adjoint functors $X\otimes_R(-)$ and $\Hom_S(X,-)$ respectively. We will need criteria to determine when this adjunction is Quillen.

\begin{lem}\label{tensor_hom_lemma}
Suppose that $X$ is projective as a left $S$-module and flat as a right $R$-module. Then the adjunction
$$
\LM_X:\Mod(R)\rightleftarrows\Mod(S):\R_X
$$
is Quillen for the Gorenstein projective model structures. Moreover, $\LM_X$ and $\R_X$ preserve weak equivalences so that there are natural isomorphisms $\mathbf{L}\LM_X(M)\cong \LM_X(M)$ and $\mathbf{R}\R_X(N)\cong \R_X(N)$ in $\Ho(R)$ and $\Ho(S)$ respectively, for all $M\in\Mod(R)$ and all $N\in\Mod(S)$, i.e.\ the derived adjunction
\begin{equation*}
\mathbf{L}\LM_X:\Ho(R)\rightleftarrows\Ho(S):\mathbf{R}\R_X
\end{equation*}
is computed without the need for any replacements. 
\end{lem}

\begin{proof}
By our assumptions on $X$, the functor $\LM_X$ is exact and preserves projectives. Therefore, $\LM_X$ preserves both Gorenstein projectives and modules of finite projective dimension. It follows that $\LM_X$ preserves both monomorphisms with Gorenstein projective cokernels (i.e.\ cofibrations) and monomorphisms with projective cokernel (i.e.\ trivial cofibrations) and so is left Quillen. By what we've shown, we therefore  have that $\LM_X$ is exact and preserves trivial objects. Furthermore, we also have that $\R_X$ is exact and that it preserves trivial objects since those are trivially fibrant and $\R_X$ is right Quillen. By applying \cite[Lemma 2.8]{Dup25}, we now immediately get that $\LM_X$ and $\R_X$ preserve all weak equivalences and that the derived adjunction is of the claimed form.
\end{proof}

In order to test whether a Quillen adjunction like in the above can be a Quillen equivalence, the following result will be useful.

\begin{lem}\label{Test_quillen_equiv} Suppose that $F:\C \rightleftarrows\D:G$ is a Quillen adjunction between model categories. If $F$ reflects and preserves weak equivalences, then the adjunction is a Quillen equivalence if and only if the counit $\varepsilon_M:FG(M)\to M$ is a weak equivalence for all fibrant objects $M$ in $\D$. Dually, if $G$ reflects and preserves weak equivalences, then the adjunction is a Quillen equivalence if and only if the unit $\eta_N: N\to GF(N)$ is a weak equivalence for all cofibrant objects $N$ in $\C$.
\end{lem}

\begin{proof}
This is just \cite[Lemma 1.6]{KD24} and its dual, whose statement assumes that the model structure on one side is right or left transferred from the other side, but whose proof applies verbatim with our assumptions.
\end{proof}

We will later need to consider Laurent polynomial rings over certain Gorenstein algebras. Such algebras are themselves automatically Gorenstein, but since we did not find an explicit reference we record it here.

\begin{lem}\label{Laurent_Gorenstein}
Suppose that $R$ is a Gorenstein ring. Then:
\begin{enumerate}
    \item the polynomial ring $R[Z]$ is Gorenstein; and
    \item any Ore localisation $\mathrm{S}^{-1}R$ at a central multiplicative set $\mathrm{S}$ is Gorenstein.
\end{enumerate}
In particular, the Laurent polynomial ring $R[Z^{\pm 1}]$ is Gorenstein.
\end{lem}

\begin{proof}
It is clear in both cases that the rings are left and right Noetherian, noting for (ii) that $\mathrm{S}^{-1}R$ is both a left and a right Ore localisation since $\mathrm{S}$ is central. The fact that $R[Z]$ has finite self-injective dimension follows immediately from \cite[Corollary 2.2]{EEI08}. For the localisation, the finite self-injective dimension property is equivalent to $\mathrm{S}^{-1}R$ having finite injective dimension as a (left and right) $R$-module. But this is immediate from the fact that $\mathrm{S}^{-1}R$ is flat over $R$ (cf.\ \cite[Theorem 8.2]{Hov2}).
\end{proof}

In the Laurent polynomial case, the next couple of results will be useful to us later on when trying to establish Quillen equivalences.

\begin{lem}\label{trivial_res_lemma}
Suppose that $S=R[Z^{\pm 1}]$ is a Laurent polynomial ring over $R$ and that $X=S$. Then the functors $\R_S=\Res^S_R$ and $\LM_S$ each preserve and reflect weak equivalences.
\end{lem}

\begin{proof}
By \Cref{tensor_hom_lemma} the two functors preserve weak equivalences. For reflection, it suffices to show that they reflect modules of finite projective dimension by \cite[Lemma 2.8(ii)]{Dup25}. For this, let $M\in\Mod(R)$ and $N\in\Mod(S)$ and assume that both $\LM_S(M)$ and $\R_S(N)$ have finite projective dimension. We have that $M$ is an $R$-linear direct summand of $\LM_S(M)=M[Z^{\pm 1}]$, which gives the statement for $\LM_S$. For $\R_S$, we apply \cite[Proposition 7.5.2]{MCR01} to get that $\pd_S(N)\leq \pd_R(N)+1<\infty$ as claimed.
\end{proof}

\begin{prop}\label{Laurent_extn}
Suppose that $X$ is an $(S,R)$-bimodule which is projective as a left $S$-module and flat as a right $R$-module. Furthermore, consider the Laurent polynomial rings $R'=R[Z^{\pm 1}]$ and $S'=S[Z^{\pm 1}]$. Let $X'\coloneqq X[Z^{\pm 1}]$ and assume that $\LM_X$ is a left Quillen equivalence. Then $\LM_{X'}:\Mod(R')\to \Mod(S')$ is also a left Quillen equivalence.
\end{prop}

\begin{proof}
First, we note that $X'\cong X\otimes_R R'$ as an $(S, R')$-bimodule and $X'\cong S'\otimes_S X$ as an $(S', R)$-bimodule. Consequently, $X'$ is flat as a right $R'$-module (cf.\ \cite[00HI]{stack}) and projective as a left $S'$-module. Hence by \Cref{tensor_hom_lemma} it follows that $\LM_{X'}$ is left Quillen. Next, by the above we have
$$
\Res^{S'}_S(\LM_{X'}(M))=\Res^{S'}_S(X'\otimes_{R'} M)\cong X\otimes_R M
$$
for all $M\in\Mod(R')$. In other words, the diagram of functors
\[
\begin{tikzcd}
\Mod(R')\arrow{r}{\LM_{X'}}\arrow[swap]{d}{\Res^{R'}_R} & \Mod(S')\arrow{d}{\Res^{S'}_S}\\
\Mod(R)\arrow{r}{\LM_X} & \Mod(S)
\end{tikzcd}
\]
commutes up to natural isomorphism. Note that our Quillen equivalence assumption and \Cref{tensor_hom_lemma} imply that $\LM_X$ both preserves and reflects trivial objects and thus weak equivalences by \cite[Lemma 2.8]{Dup25}. But the two restriction functors also both preserve and reflect weak equivalences by \Cref{trivial_res_lemma}, hence so does $\LM_{X'}$ from the commutativity of the diagram. \Cref{Test_quillen_equiv} therefore applies, so that we must check that the counit $\varepsilon_M:\LM_{X'}\R_{X'}(N)\to N$ is a weak equivalence for all $N\in\Mod(S')$.

For this, note that by the tensor-hom adjunction there is a natural isomorphism $\Hom_S(X,\Res_S^{S'}(-))\cong \Hom_{S'}(X',-)$. Viewing both sides as functors valued in $\Mod(R)$, this says that there is a natural isomorphism $\Res^{R'}_R\circ\R_{X'}\to\R_X\circ \Res^{S'}_S$, and hence that we have a commuting square as above for the right adjoints as well. It follows that
$$
\Res^{S'}_S(\LM_{X'}\R_{S'}(N))\cong \LM_X\R_X(\Res^{S'}_S(N))
$$
for all $N\in \Mod(S')$. By chasing down these identifications we obtain that the restriction to $S$ of $\varepsilon_M$ is the counit of the adjunction $\LM_X\vdash\R_X$. But by assumption and \Cref{Test_quillen_equiv}, the latter is a weak equivalence in $\Mod(S)$. Since $\Res^{S'}_S$ reflects weak equivalences by \Cref{trivial_res_lemma}, this implies that $\varepsilon_M$ is a weak equivalence as well and we are done.
\end{proof}

\subsubsection{Decomposition of $\Ho(R)$ by a central idempotent} Suppose that we are given a non-trivial central idempotent $e\in R$. This situation will be relevant for when considering Hecke algebras, cf.\ \eqref{decomp_H}. Then the decomposition $R=eR\times (1-e)R$ of rings induces a decomposition $\Mod(R)=\Mod(eR)\times \Mod((1-e)R)$. Since $\Hom$ sets and thus $\Ext$ groups decompose similarly into a product, we deduce that $eR$ and $(1-e)R$ are both Gorenstein algebras (of self-injective dimension bounded above by that of $R$) and that we have a decomposition
\begin{equation}\label{idem_decomp}
\Ho(R)=\Ho(eR)\times \Ho((1-e)R).
\end{equation}
We will say that an $R$-module \emph{factors through} $eR$ if it is annihilated by $1-e$. If $M, N\in\Mod(R)$ factor through $eR$, we see from \eqref{idem_decomp} that $[M,N]_R=[M,N]_{eR}$.

\section{The category $\Ho(\He)$ in rank one}

In this section, we prove that $\Ho(\He)$ is equivalent to the homotopy category of some explicit commutative Gorenstein rings. We also describe these categories completely explicitly when $\G=\SL_2$ or $\PGL_2$.

\subsection{Homotopy categories in terms of facets in the Bruhat-Tits building} We begin with a general discussion that applies for $\G$ an arbitrary split connected reductive group. As this will play a role in the sequel, we first recall that, given an adjunction $F:\C\rightleftarrows\D:U$ between the underlying categories of two model categories, then we say that the model structure on $\C$ is the \emph{left transfer} along $F$ of the model structure on $\D$ when a morphism $f$ in $\C$ is a cofibration, resp.\ weak equivalence, if and only if $Ff$ is a cofibration, resp.\ weak equivalence, in $\D$. Dually, we say that the model structure on $\D$ is the \emph{right transfer} along $U$ of the model structure on $\C$ when a morphism $f$ in $\D$ is a fibration, resp.\ weak equivalence, if and only if $Uf$ is a fibration, resp.\ weak equivalence, in $\C$.

For each $0\leq i\leq d=r_{\text{ss}}$, let $\mathcal{F}_{(i)}$ denote a set of representatives for the $G$-orbits of the $i$-dimensional faces contained in $\overline{C}$ in $\mathscr{X}$. Recall that the Ollivier-Schneider resolution of a Hecke module $M\in\Mod(\He)$ is of the form
\begin{equation}\label{OllSch_res_gen}
0\longrightarrow \bigoplus_{F\in\mathcal{F}_{(d)}} \He(\epsilon_F)\otimes_{\He^\dagger_F}M\longrightarrow\cdots\longrightarrow \bigoplus_{F\in\mathcal{F}_{(0)}} \He(\epsilon_F)\otimes_{\He^\dagger_F}M\stackrel{\varepsilon}{\longrightarrow} M\longrightarrow 0,
\end{equation}
where $\varepsilon$ is the map induced by $h\otimes m\mapsto h\cdot m$, cf.\ \cite[Theorem 3.12]{OS14}.

In what follows, we will view $\He$ as an $(\He, \He_F^\dagger)$-bimodule for each $F\subseteq\overline{C}$ and will write $\LM^F_{\He}$ for the tensor functor $\He\otimes_{\He_F^\dagger}(-):\Mod(\He_F^\dagger)\to \Mod(\He)$ in order to distinguish between faces. Using that $\He$ is free over $\He_F^\dagger$ for any $F\subseteq \overline{C}$, cf.\ \cite[Proposition 4.21]{OS14}, we see from \Cref{tensor_hom_lemma} that $\LM^F_{\He}$ is left Quillen and that $\LM^F_{\He}$ and $\Res^\He_{\He_F^\dagger}$ each both preserve and reflect weak equivalences. In fact, something stronger holds as we now see.

\begin{lem}\label{right_transf} We let $\mathcal{F}\coloneqq \bigcup_{0\leq i\leq d-1}\mathcal{F}_{(i)}$ and $\mathcal{F}_C=\mathcal{F}\cup\{C\}$. The Gorenstein projective model structure on $\Mod(\He)$ is the right transfer along the `diagonal functor'
\begin{equation}
\Delta_C\coloneqq \left(\Res^\He_{\He_F^\dagger}\right)_{F\in\mathcal{F}_C}\colon\Mod(\He)\to \prod_{F\in\mathcal{F}_C} \Mod(\He_F^\dagger)
\end{equation}
of the product Gorenstein projective model structures on the right hand side (i.e.\ with all classes of morphisms defined termwise). Furthermore, the left adjoint functor
$$
L\colon\left(M_F\right)_{F\in\mathcal{F}_C}\longmapsto\bigoplus_{F\in\mathcal{F}_C} \He\otimes_{\He_F^\dagger}M_F
$$
both preserves and reflects weak equivalences. Assuming that $p\nmid |\Omega_{\text{tor}}|$, theses statements hold as well for the functor $\Delta\coloneqq \left(\Res^\He_{\He_F^\dagger}\right)_{F\in\mathcal{F}}$ and its left adjoint. 
\end{lem}

\begin{proof}
The proof that the model structure is right transferred through $\Delta$ was done in \cite[Lemma 5.6(iii)]{Dup25} for the affine Hecke algebra and the finite Hecke algebras $\He_F$, and it applies verbatim here (one only needs to appeal to \cite[Lemma 4.2]{Koz} in order to get reflection of trivial objects).

For the statement on the left adjoint, note that $L\left(M_F\right)_{F}=\bigoplus_F\LM^F_{\He}(M_F)$ by definition. Since weak equivalences in the product category $\prod_{F\in\mathcal{F}_C} \Mod(\He_F^\dagger)$ are defined termwise, and as each $\LM^F_{\He}$ preserves weak equivalences by our discussion above, it follows that $L$ preserves weak equivalences. Conversely, assume that $L(f_F)=\bigoplus_F\LM^F_{\He}(f_F)$ is a weak equivalence. Since weak equivalences are preserved under retracts and since each $\LM^F_{\He}$ reflects weak equivalences by the above discussion, it follows that each $\LM^F_{\He}(f_F)$ is a weak equivalence and hence so is $f_F$. Thus $L$ reflects weak equivalences as claimed.

Finally, we can omit $\mathcal{F}_{(d)}=\{C\}$ in the definition of $\Delta$ when $p\nmid |\Omega_{\text{tor}}|$ because the algebra $\He_C^\dagger=k[\widetilde{\Omega}]$ is an iterated skew Laurent polynomial algebra over the group algebra $k[\widetilde{\Omega}_{\text{tor}}]$ (cf.\ the proof of \cite[Lemma 3.4]{Koz}). When $p\nmid |\Omega_{\text{tor}}|$, the latter is semisimple and hence $\He_C^\dagger$ has then finite global dimension (cf.\ \cite[Theorem 7.5.3(ii)]{MCR01}).
\end{proof}

\begin{rem}
While we won't need this, the functor $\LM^F_{\He}$ in fact reflects cofibrant objects, i.e.\ Gorenstein projective modules. Indeed, this holds because the map $M\to \He\otimes_{\He_F^\dagger} M$, $m\mapsto 1\otimes m$, is a split monomorphism for any $M\in \Mod(\He_F^\dagger)$ by \cite[Proposition 4.21(ii)]{OS14}. Thus one deduces that the Gorenstein projective model structure on $\Mod(\He_F^\dagger)$ is the left transfer through $\LM_\He^F$ of the Gorenstein projective model structure on $\Mod(\He)$. Similarly, the product model structure on $\prod_{F\in\mathcal{F}_C} \Mod(\He_F^\dagger)$ is the left transfer through $L$ of the Gorenstein projective model structure on $\Mod(\He)$.
\end{rem}

We now apply the above result to our rank 1 setting. In the case that $\G=\GL_2$ or $\SL_2$, we note that $\Omega_{\text{tor}}=1$ and the condition on $p$ above is empty, and when $\G=\PGL_2$ we have $\Omega_{\text{tor}}=\Z/2\Z$ so that we are only excluding $p=2$.

\begin{cor}\label{Delta_equiv}
Suppose that $\G=\mathrm{SL_2}$, $\mathrm{PGL_2}$ or $\mathrm{GL_2}$, and assume that $p\neq 2$ when $\G=\mathrm{PGL_2}$. Then the functor $\Delta$ from \Cref{right_transf} is a Quillen equivalence.
\end{cor}

\begin{proof}
By \Cref{right_transf} and \Cref{Test_quillen_equiv}, it suffices to show that the counit of the adjunction $L\dashv \Delta$ is a weak equivalence. But this counit is precisely the map $\varepsilon$ from the Ollivier-Schneider resolution \eqref{OllSch_res_gen}, which is now a short exact sequence, and is therefore a weak equivalence by \cite[Lemma 5.8]{Hov2} since our condition on $p$ ensures that $\He_C^\dagger$ has finite global dimension.
\end{proof}

Explicitly, \Cref{Delta_equiv} tells us that $\Ho(\Res^\He_{\He_{x_0}^\dagger}):\Ho(\He)\to \Ho(\He_{x_0}^\dagger)$ is an equivalence when $\G=\mathrm{PGL_2}$ or $\mathrm{GL_2}$ (and for $\mathrm{PGL_2}$ one has $\He_{x_0}^\dagger=\He_{x_0}$), and that $\Ho(\Delta):\Ho(\He_{\SL_2})\to \Ho(\He_{\SL_2,x_0})\times \Ho(\He_{\SL_2,x_1})$ is an equivalence.

\begin{rem}
The analogue of \Cref{Delta_equiv} will not hold in higher rank. Indeed, the induced functor
$$
\Ho(\Delta):\Ho(\He)\to\prod_{F\in\mathcal{F}}\Ho(\He_F^\dagger)
$$
will then not be full, since the image of $[M,N]_\He$ under it is certainly contained in the set of all tuples $(f_F)_{F\in\mathcal{F}}$ with $f_F\in[M,N]_{\He_F^\dagger}$ such that the $f_F$ are suitably compatible under restrictions. We mean by this that if $\omega\in\Omega$ such that $\omega F'\subseteq \overline{F}$, then we can pullback $f_{F'}$ to a morphism $\omega\cdot f_{F'}\in[M,N]_{\He_{\omega F'}^\dagger}$ and this should at least satisfy that $\Res^{\He_{\omega F'}^\dagger}_{\He_F}(\omega\cdot f_{F'})$ is homotopic to $\Res^{\He_F^\dagger}_{\He_F}(f_{F})$. While such compatibility conditions are empty in rank one since $\Ho(\He_C)=0$, they become non-trivial in higher rank.
\end{rem}

Recall from \cite[Proposition 4.10]{KD24} that the categories $\Rep_k^\infty(G)$ and $\Rep_k^\infty(P_F^\dagger)$ (for any $F\subseteq \overline{C}$) of smooth $k$-linear representations have a so called $I$-Gorenstein projective model structure, which is defined as the right transfer of the Gorenstein projective model structure along the functors of $I$-invariants
$$
(-)^I:\Rep_k^\infty(G)\to \Mod(\He)\quad\text{and}\quad (-)^I:\Rep_k^\infty(P_F^\dagger)\to \Mod(\He_F^\dagger).
$$
The left adjoint of these functors are given by $\XX\otimes_{\He}(-)$ and $\XX_F^\dagger\otimes_{\He_F^\dagger}(-)$ respectively, where $\XX\coloneqq \cind_I^G(\mathbf{1})$ and $\XX_F^\dagger\coloneqq \cind_I^{P_F^\dagger}(\mathbf{1})$. In rank one, we have an analogue of our Hecke result which reduces the study of the homotopy category $\Ho(\Rep_k^\infty(G))$ to studying $\Ho(P_x^\dagger)$ with $x\in\mathcal{F}$.

\begin{prop}\label{Delta_equiv_reps}
Suppose that $\G=\mathrm{SL_2}$, $\mathrm{PGL_2}$ or $\mathrm{GL_2}$, and assume that $p\neq 2$ when $\G=\mathrm{PGL_2}$. Then the functor
$$
\tilde{\Delta}\coloneqq \left(\Res^G_{P_x^\dagger}\right)_{x\in\mathcal{F}}:\Rep_k^\infty(G)\to\prod_{x\in\mathcal{F}}\Rep_k^\infty(P_x^\dagger)
$$
is a right Quillen equivalence with respect to the $I$-Gorenstein projective model structures.
\end{prop}

\begin{proof}
We first claim that the aforementioned model structure on $\Rep_k^\infty(G)$ is the right transfer along $\tilde{\Delta}$ of the product model structure on $\prod_{x\in\mathcal{F}}\Rep_k^\infty(P_x^\dagger)$. To see this, first of all, the functor $\tilde{\Delta}$ has left adjoint given by $(V_x)_{x\in\mathcal{F}}\mapsto \bigoplus_{x\in\mathcal{F}}\cind_{P_x^\dagger}^G(V_x)$. Let now $f$ be a morphism in $\Rep_k^\infty(G)$. By definition of the model structure on $\Rep_k^\infty(G)$, we have that $f$ is a fibration (resp.\ weak equivalence) if and only if $f^I$ is a fibration (resp.\ weak equivalence) in $\Mod(\He)$. But by \Cref{right_transf}, the latter holds if and only if $\Delta(f^I)$ is a fibration (resp.\ weak equivalence) in $\prod_{x\in\mathcal{F}}\Mod(\He_x^\dagger)$. Since the square
\[
\begin{tikzcd}
\Rep_k^\infty(G) \arrow{r}{\tilde{\Delta}} \arrow[swap]{d}{(-)^I} & \prod_{x\in\mathcal{F}}\Rep_k^\infty(P_x^\dagger) \arrow{d}{(-)^I}\\
\Mod(\He) \arrow{r}{\Delta}& \prod_{x\in\mathcal{F}}\Mod(\He_x^\dagger)
\end{tikzcd}
\]
commutes, the definition of the model structure on $\prod_{x\in\mathcal{F}}\Rep_k^\infty(P_x^\dagger)$ gives that $\Delta(f^I)$ is a fibration (resp.\ weak equivalence) if and only if $\tilde{\Delta}(f)$ is a fibration (resp.\ weak equivalence). This shows our claim. By \Cref{Test_quillen_equiv}, we now see that it suffices to show that the unit
\begin{equation}\label{unit}
(V_x)_{x\in\mathcal{F}}\to \tilde{\Delta}\left(\bigoplus_{x\in\mathcal{F}}\cind_{P_x^\dagger}^G(V_x)\right)
\end{equation}
is a weak equivalence for all (termwise) cofibrant objects $(V_x)_{x\in\mathcal{F}}$.

We now fix such a collection $(V_x)_{x\in\mathcal{F}}$. For each $y\in \mathcal{F}$, the $y$-component of \eqref{unit} above is just the canonical $P_y^\dagger$-equivariant inclusion
\begin{equation}\label{component_unit}
V_y\hookrightarrow \cind_{P_y^\dagger}^G(V_y)\subseteq \bigoplus_{x\in\mathcal{F}}\cind_{P_x^\dagger}^G(V_x)
\end{equation}
Since each $V_x$ is cofibrant, one has in particular that $V_x=k[P_x^\dagger]\cdot V_x^I$ is generated by its $I$-invariants (cf.\ \cite[Lemma 4.14]{KD24}). Note therefore that since $I_x\leq I$ is normal in $P_x^\dagger$, it must then acts trivially on $V_x$. Applying \cite[Proposition 4.17]{Koh}, we now get that taking $I$-invariants in \eqref{component_unit} yields the canonical map
$$
V_y^I\to \He\otimes_{\He_y^\dagger}V_y^I\subseteq \bigoplus_{x\in\mathcal{F}}\He\otimes_{\He_x^\dagger}V_x^I,
$$
which is just the $y$-component of the unit $\eta_{(V_x^I)_{x}}$ of the adjunction $L\vdash \Delta$ from \Cref{right_transf}. The statement that \eqref{unit} is a weak equivalence therefore amounts to showing that $\eta_{(V_x^I)_{x}}$ is a weak equivalence.

Putting together \Cref{Delta_equiv} and \Cref{Test_quillen_equiv}, we see that $\eta_{(M_x)_x}$ is a weak equivalence whenever each $M_x\in\Mod(\He_x^\dagger)$ is Gorenstein projective. Choose such a collection $(M_x)_{x\in\mathcal{F}}$ which is a cofibrant replacement of $(V_x^I)_{x\in\mathcal{F}}$, i.e.\ such that there is a short exact sequence $0\to (K_x)\to (M_x)\to (V_x^I)\to 0$ where each $K_x\in\Mod(\He_x^\dagger)$ has finite projective dimension. By exactness of $\Delta$ and $L$, and by naturality of $\eta$, we have a commutative diagram
\[
\begin{tikzcd}
0\arrow{r} & (K_x) \arrow{r}\arrow{d}{\eta_{(K_x)}} & (M_x) \arrow{r}\arrow{d}{\eta_{(M_x)}} & (V_x^I) \arrow{r}\arrow{d}{\eta_{(V_x^I)}} & 0\\
0\arrow{r} & \Delta L(K_x) \arrow{r} & \Delta L(M_x) \arrow{r} & \Delta L(V_x^I) \arrow{r} & 0
\end{tikzcd}
\]
and applying the snake lemma now gives a short exact sequence
$$
0\longrightarrow \coker(\eta_{(K_x)})\longrightarrow \coker(\eta_{(M_x)})\longrightarrow \coker(\eta_{(V_x^I)})\longrightarrow 0\,.
$$
Since $\Delta$ and $L$ both preserve trivial objects (cf.\ \Cref{right_transf}), it follows that $\coker(\eta_{(K_x)})$ is a trivial object, i.e.\ termwise of finite projective dimension. Further, the fact that $\eta_{(M_x)}$ is a weak equivalence implies that $\coker(\eta_{(M_x)})$ is also trivial (cf.\ \cite[Lemma 5.8]{Hov2}). Hence, a long exact sequence argument gives that $\coker(\eta_{(V_x^I)})$ is termwise of finite projective dimension and another application of \textit{loc.\ cit}.\ gives that $\eta_{(V_x^I)}$ is a weak equivalence as required.
\end{proof}

For example in the case of $G=\GL_2(\Ff)$, a heuristic for why the above result should hold is that a smooth $G$-representation $\pi$ over $k$ is uniquely determined by the corresponding constant diagram $(\pi|_{P_{x_0}^\dagger}, \pi|_{P_{C}^\dagger}, \id_\pi)$ in the sense \cite[\S 5.5]{Pas04}. Since $\Ho(\Rep_k^\infty(P_C^\dagger))=0$, it follows that this diagram shrinks to the singleton $\pi|_{P_{x_0}^\dagger}$ in our homotopy language, and so that $\pi$ should be determined up to isomorphism in $\Ho(\Rep_k^\infty(G))$ by its restriction to $P_{x_0}^\dagger=\GL_2(\OF)\Ff^\times$.

\subsection{The block decomposition of $\Ho(\He)$ and spherical modules} In order to describe $\Ho(\He)$ more explicitly, it will be useful to use a certain decomposition by idempotents of the Hecke algebra. Recall that for each character $\xi:T(\F_q)\to k^\times$ there is an associated idempotent
\[
e_\xi\coloneqq \frac{1}{|T(\F_q)|}\sum_{t\in T(\F_q)} \xi(t^{-1})T_t\in k[T(\F_q)]= \He_C\,.
\]
Following the terminology of Vignéras in \cite[\S 2-3]{Vign_GL2}, we call the character $\xi$ \emph{regular} when $\xi\neq\xi^{s_0}$ and \emph{non-regular} otherwise. The set $T(\F_q)^\vee/W_0$ of $W_0$-orbits of $k$-valued characters of $T(\F_q)$ naturally decomposes as the disjoint union $(T(\F_q)^\vee/W_0)_{\text{non-reg}}\bigsqcup (T(\F_q)^\vee/W_0)_{\reg}$ of the orbits of non-regular and regular characters respectively. Given $\gamma\in T(\F_q)^\vee/W_0$, we will then write
$$
e_\gamma\coloneqq \begin{cases}
e_\xi & \text{if $\gamma=\{\xi\}\in (T(\F_q)^\vee/W_0)_{\text{non-reg}}$}\\
e_\xi+e_{\xi^{s_0}} & \text{if $\gamma=\{\xi, \xi^{s_0}\}\in (T(\F_q)^\vee/W_0)_{\text{reg}}$}
\end{cases}.
$$
Note that the braid relations imply that $e_\gamma$ is central in $\He$ for all $\gamma\in T(\F_q)^\vee/W_0$. We thus have product decompositions
\begin{equation}\label{decomp_H}
\He= \prod_{\gamma\in T(\F_q)^\vee/W_0} e_\gamma\He\quad \text{and} \quad \Hea= \prod_{\gamma\in T(\F_q)^\vee/W_0} e_\gamma\Hea\,.
\end{equation}
We point out that this decomposition of $\He$ is compatible with the Ollivier-Schneider resolutions, in the sense that for any $\gamma\in T(\F_q)^\vee/W_0$ and any $M\in\Mod(e_\gamma\He)$ the resolution \eqref{OllSch_res_gen} induces and exact sequence
\begin{equation}\label{OllSch_res_comp}
0\longrightarrow e_\gamma\He(\epsilon_C)\otimes_{e_\gamma\He_C^\dagger}M\longrightarrow\bigoplus_{x\in\mathcal{F}_{(0)}} e_\gamma\He\otimes_{e_\gamma\He_x^\dagger}M\stackrel{\varepsilon}{\longrightarrow}M\longrightarrow 0\,.
\end{equation}
Next, the decomposition \eqref{decomp_H} implies that the homotopy categories decompose into products
\begin{equation}\label{decomp_Ho_H}
\Ho(\He)=\prod_{\gamma\in T(\F_q)^\vee/W_0} \Ho(e_\gamma\He)\quad \text{and} \quad \Ho(\Hea)=\prod_{\gamma\in T(\F_q)^\vee/W_0} \Ho(e_\gamma\Hea)
\end{equation}
as well, cf.\ \eqref{idem_decomp}. Similarly, for $i=1,2$ we have decompositions
\[
\He_{x_i}\coloneqq \prod_{\gamma\in T(\F_q)^\vee/W_0} e_\gamma\He_{x_i}\quad \text{and} \quad \He_{x_i}^\dagger\coloneqq \prod_{\gamma\in T(\F_q)^\vee/W_0} e_\gamma\He_{x_i}^\dagger\,,
\]
and correspondingly again
\begin{equation}\label{decomp_Ho_Hx}
\Ho(\He_{x_i})=\prod_{\gamma\in T(\F_q)^\vee/W_0} \Ho(e_\gamma\He_{x_i}) \quad \text{and} \quad\Ho(\He_{x_i}^\dagger)=\prod_{\gamma\in T(\F_q)^\vee/W_0} \Ho(e_\gamma\He_{x_i}^\dagger)\,.
\end{equation}
We can however ignore some of the factors above by the following result.

\begin{lem}\label{non_reg_is_reg}
Suppose that $\xi|_{\alpha^\vee(\F_q^\times)}=1$. Then $\xi$ is non-regular and the algebras $e_\xi\He_{x_i}$, $e_\xi\He_{x_i}^\dagger$ (for $i=1,2$) and $e_\xi \Hea$ have finite global dimension. Moreover, if $p\nmid |\Omega_{\text{tor}}|$ then $e_\xi\He$ also has finite global dimension.
\end{lem}

\begin{proof}
The non-regularity of $\xi$ is follows directly from the identity
\begin{equation}\label{eq_xi}
\xi\begin{pmatrix}
x & 0\\
0 & y
\end{pmatrix}=\xi (\alpha^\vee(xy^{-1}))\cdot \xi^{s_0}\begin{pmatrix}
x & 0\\
0 & y
\end{pmatrix}\,,
\end{equation}
which holds for any $x, y\in\F_q^\times$ and any character $\xi$. Next, to show $e_\xi\Hea$ has finite global dimension, we appeal to \cite[Lemma 4.2]{Koz} and the Gorenstein property of $\Hea$ (cf.\ the paragraph following \cite[Proposition 5.1]{Dup25}) to see that it suffices to show that each $e_\xi\He_{x_i}$ has finite global dimension. But each $e_\xi\He_{x_i}$ is isomorphic to $k[X]/(X^2+X)\cong k^2$ and so even has global dimension zero. It remains to show that $e_\xi\He$ has finite global dimension under our assumption on $p$. We may argue just as for $\Hea$ (appealing to \Cref{right_transf}) to see that we only need to show the result for $e_\xi\He_{x_i}^\dagger$. But we have $e_\xi\He_{x_i}^\dagger=e_\xi\He_{x_i}$ for $\G=\SL_2$ or $\PGL_2$ and $e_\xi\He_{x_i}^\dagger=e_\xi\He_{x_i}[T_{\omega^2}^{\pm 1}]$ for $\G=\GL_2$ (for $i=1,2$). In either case, it follows that $e_\xi\He_{x_i}^\dagger$ has finite global dimension from the above and we are done.
\end{proof}

\begin{rem}\label{non-reg}
We deduce from \eqref{eq_xi} that $\xi|_{\alpha^\vee(\F_q^\times)}=1$ if and only if $\xi$ is non-regular when $\G=\PGL_2$ or $\GL_2$. When $\G=\SL_2$, \eqref{eq_xi} implies instead that $\xi$ is non-regular if and only if $\xi|_{\alpha^\vee(\F_q^\times)^2}=1$. For $p=2$, the latter is equivalent to $\xi$ being trivial. For $p>2$, there is a unique non-trivial and non-regular character, namely the `sign' character $\sigma$ which sends any fixed generator of $\F_q^\times\cong  T(\F_q)$ to $-1$.
\end{rem}

Our next goal is to relate $\Ho(\He)$ to the centre $Z(\He)$ of $\He$ via the so-called spherical module. This will play an important role in \Cref{section_sing}. As a preliminary observation, we have from \eqref{decomp_H} that $Z(\He)$ also decomposes as
$$
Z(\He)=\prod_{\gamma\in T(\F_q)^\vee/W_0} Z(e_\gamma\He)\,,
$$
so that we will just need to work component-wise. We first give a description of the components $e_\gamma\He$ in terms of their centre for each of our groups.

\subsubsection{The case $\G=\GL_2$} The components $Z(e_\gamma\He_{\GL_2})$ were all determined by Vignéras in \cite[\S 2.1]{Vign_GL2} and we recall the results. When $\gamma=\{\xi\}$ is non-regular then $e_\xi\He$ is the $k$-algebra with unit $e_\xi$, generators $e_\xi T_{s_0}$ and $e_\xi T_\omega^{\pm 1}$, and relations $e_\xi T_{s_0}^2=-e_\xi T_{s_0}$ and $e_\xi T_{\omega}^2 T_{s_0}=e_\xi T_{s_0}T_{\omega}^2$. It is then shown in section 1.2 of \textit{loc.\ cit}.\ that $k[X, Z^{\pm 1}]\cong Z(e_\xi\He)$ with $Z\mapsto e_\xi T_{\omega}^{2}$ and $X\mapsto e_\xi (T_{s_0}+1)T_\omega+e_\xi T_{\omega s_0}$. In fact, by \cite[Corollary 4.3.5]{PeSch23} there is an embedding of $k$-algebras $e_\xi\He\hookrightarrow M_2(k[X, Z^{\pm 1}])$ defined by
\begin{equation}\label{sph_mod_nonreg}
e_\xi T_\omega\mapsto\begin{pmatrix}
    X & X^2 -Z\\
    -1 & -X
\end{pmatrix}\quad\text{and}\quad e_\xi T_{s_0}\mapsto \begin{pmatrix}
    0 & 0\\
    0 & -1
\end{pmatrix}\,.
\end{equation}
The centre $Z(e_\xi\He)$ identifies with the scalar matrices under this embedding. Following the terminology of \textit{loc.\ cit}.\ we will call the induced action of $e_\gamma\He$ on $Z(e_\gamma\He)^2$ the \emph{spherical $e_\gamma\He$-module} and denote it by $\M_{\GL_2,\gamma}$.

The case when $\gamma$ is regular is of main interest for us, cf.\ \Cref{non_reg_is_reg} and \Cref{non-reg}, so that we give more details. In what follows we will set
$$
A\coloneqq k[X_1, X_2]/(X_1X_2)\quad\text{and}\quad B\coloneqq A[Z^{\pm 1}]=k[X_1, X_2, Z^{\pm 1}]/(X_1X_2)\,.
$$
Now let $\gamma=\{\xi, \xi^{s_0}\}$ be regular and write for simplicity $e_1=e_\xi$ and $e_2=e_{\xi^{s_0}}$. By the quadratic relations we have $e_\gamma T_{s_0}^2=0$. Consequently, we have that $e_1T_{s_0\omega}$ and $e_1T_{s_1\omega}=e_1T_{\omega s_0}$ have product in either order equal to zero.

We now let $D_0$ and $D_1$ denote the sets of elements of $W_\aff$ with reduced expression starting with $s_0$ and $s_1$ respectively. For each $n\geq 1$ and $i=0,1$, there is in particular a unique $w_{i,n}\in D_i$ of length $n$. One then straightforwardly computes that
\begin{equation}\label{span_A}
\left(e_1T_{s_i\omega}\right)^n=e_1 T_{w_{i,n}\omega^{ n}}
\end{equation}
for all $n\geq 1$ and $i=0,1$. If we let $\A$ denote the $k$-span of $e_1$, $(e_1T_{s_0\omega})^n$ and $(e_1T_{s_1\omega})^n$ (for all $n\geq 1$), then it follows that $\A$ is a $k$-algebra with unit $e_1$ under the multiplication in $\He$ and that there is an isomorphism $\A\cong A$ defined by $e_1T_{s_0\omega}\mapsto X_1$ and $e_1T_{s_1\omega}\mapsto X_2$. Similarly, $\B\coloneqq \bigoplus_{n\in\Z} \A\cdot T_{\omega^{2n}}\subseteq e_1\He$ is also a $k$-algebra with unit $e_1$, and the above extends to an isomorphism $\B\cong B=A[Z^{\pm 1}]$. Note in fact from \eqref{span_A} that $\B$ is the $k$-span of the $e_1T_{w\omega^n}$ such that $w\in W_\aff$ satisfies $\ell(w)\equiv n \mod{2}$. From the braid relations, we therefore have
\[
e_1\He=\B\oplus \B T_\omega \quad\text{and}\quad e_2\He=T_\omega\B\oplus T_\omega\B T_{\omega^{-1}}\,.
\]
From this one gets the following result (cf.\ \cite[Corollary 2.3]{Vign_GL2} or \cite[Proposition 3.3.3]{PeSch23}).

\begin{prop}\label{matrix_description_H}
Let $\G=\GL_2$ and $\gamma\in (T(\F_q)^\vee/W_0)_{\reg}$. There is a $k$-algebra isomorphism
$$
e_\gamma\He \stackrel{\cong}{\longrightarrow} M_2(B)
$$
to the $k$-algebra of $2\times 2$-matrices over $B$, defined by \(e_\gamma T_\omega\mapsto\begin{psmallmatrix}
0 & Z\\
1 & 0
\end{psmallmatrix}\) and \(h\mapsto \begin{psmallmatrix}
\varphi(h) & 0\\
0 & 0
\end{psmallmatrix}\) for all $h\in \B$, where $\varphi$ denotes the isomorphism $\B\cong B$ from the above. The centre $Z(e_\gamma\He)$ corresponds under this map to scalar matrices and thus $Z(e_\gamma\He)\cong k[X_1, X_2, Z^{\pm 1}]/(X_1X_2)$.
\end{prop}

Concretely, the above isomorphism sends $e_1$ to \(\begin{psmallmatrix}
1 & 0\\
0 & 0
\end{psmallmatrix}\), $e_2$ to \(\begin{psmallmatrix}
0 & 0\\
0 & 1
\end{psmallmatrix}\) and $e_\gamma T_{s_0}$ to \(\begin{psmallmatrix}
0 & X_1\\
X_2Z^{-1} & 0
\end{psmallmatrix}\). Furthermore, we deduce that there is a natural $e_\gamma \He$-module structure on $Z(e_\gamma\He)^2$ via this map. We again call this module the \emph{spherical} $e_\gamma \He$-module and denote it by $\M_{\GL_2,\gamma}$. The \emph{spherical} $\He$-module $\M$ is the module which under \eqref{decomp_H} decomposes as $\M_{\GL_2}=\bigoplus_\gamma \M_{\GL_2,\gamma}$.

Both in the regular and non-regular case, the algebra $Z(e_\gamma \He)$ is Gorenstein by our explicit descriptions and hence so is $Z(\He)$. In particular, the category $Z(\He)$-modules has a Gorenstein projective model structure. Moreover, the algebra $Z(e_\gamma \He)$ is of infinite global dimension if and only if $\gamma$ is regular\footnote{It is an unfortunate clash of conventions that $Z(e_\gamma \He)$ is a regular commutative ring if and only if $\gamma$ is non-regular.}. Thus we have
$$
\Ho(Z(\He))\simeq \prod_{\text{$\gamma$ regular}} \Ho(Z(e_\gamma \He)),
$$
mirroring the situation for $\Ho(\He)$, cf. \Cref{non_reg_is_reg} and \Cref{non-reg}. In view of the above results, we now recall the following well-known fact.

\begin{lem}[{\cite[Theorem 17.20]{Lam99}}]\label{Morita_equiv} Let $R$ be a ring and $n\geq 1$, and view $R^n$ an as $(M_n(R), R)$-bimodule. Then the functor $\LM_{R^n}:\Mod(R)\to \Mod(M_n(R))$ is an exact equivalence of categories.
\end{lem}

As an immediate consequence, we obtain the following.

\begin{cor}\label{Quillen_equiv_GL2} Let $G=\GL_2(\Ff)$ and $\M$ be the spherical $\He$-module. The functor
$$
M\mapsto \M\otimes_{Z(\He)} M
$$
defines both a left and right Quillen equivalence $\Mod(Z(\He))\to \Mod(\He)$.
\end{cor}

\begin{proof}
From the above discussion and by \Cref{non_reg_is_reg} and \Cref{non-reg}, it suffices to show that
$$
\LM_{\M_\gamma}:\Mod(Z(e_\gamma\He))\to \Mod(e_\gamma\He)
$$
is both a left and a right Quillen equivalence when $\gamma$ is regular. This now follows immediately from \Cref{Morita_equiv}.
\end{proof}

\subsubsection{The case $\G=\PGL_2$} We can now use the above to investigate $\Ho(\He)$ in the case $\G=\PGL_2$. We first pick a character $\xi$ of $T_{\PGL_2}(\F_q)$. Via the natural surjection $\mathbf{T}_{\GL_2}\to\mathbf{T}_{\PGL_2}$ this lifts to a character of $T_{\GL_2}(\F_q)$ and by slight abuse of notation we denote the Weyl group orbits of both of these characters by $\gamma$. Then the canonical surjection $\He_{\GL_2}\to \He_{\PGL_2}$ induces a surjection $e_\gamma\He_{\GL_2}\to e_\gamma\He_{\PGL_2}$. Since $e_\gamma T_t=e_\gamma$ for all $t\in \F_q^\times\cdot \id$, we deduce that
\[
e_\gamma\He_{\PGL_2}\cong e_\gamma\He_{\GL_2}/(e_\gamma T_\omega^2-1)\cong e_\gamma\He_{\GL_2}\otimes_{Z(e_\gamma\He_{\GL_2})}Z(e_\gamma\He_{\GL_2})/(e_\gamma T_\omega^2-1)\,.
\]
For completeness, we next include a description of the spherical module.

\begin{prop}\label{matrix_description_H_PGL_2}
Let $\G=\PGL_2$ and $\gamma\in T(\F_q)^\vee/W_0$.
\begin{enumerate}
    \item If $\gamma=\{\xi\}$ is non-regular, there is an injective $k$-algebra homomorphism $e_\xi\He\to M_2(k[X])$ defined by
    \[
    e_\xi T_\omega\mapsto\begin{pmatrix}
    X & X^2 -1\\
    -1 & -X
\end{pmatrix}\quad\text{and}\quad e_\xi T_{s_0}\mapsto \begin{pmatrix}
    0 & 0\\
    0 & -1
\end{pmatrix}\,.
    \]
    This map induces an isomorphism between the centre and scalar matrices, and so $Z(e_\gamma\He)\cong k[X]$.
    \item If $\gamma$ is regular, there is an isomorphism $e_\gamma\He\to M_2(A)$ of $k$-algebras and thus $Z(e_\gamma\He)\cong A$.
\end{enumerate}
\end{prop}

\begin{proof}
By applying the functor $(-)\otimes_{Z(e_\gamma\He_{\GL_2})}Z(e_\gamma\He_{\GL_2})/(e_\gamma T_\omega^2-1)$ to the spherical representations in \eqref{sph_mod_nonreg} and \Cref{matrix_description_H}, we get a $k$-algebra homomorphism of the claimed form in both cases and we immediately deduce (ii) because of the isomorphism in \Cref{matrix_description_H} and using $M_2(B)/(Z-1)\cong M_2(A)$.

For (i), we need to show that the map is indeed injective. For this, we note that the elements $1$, $e_\xi(T_\omega T_{s_0})^n$, $e_\xi(T_{s_0}T_\omega)^n$, $e_\xi T_\omega (T_{s_0}T_\omega)^{n-1}$ and $e_\xi T_{s_0}(T_\omega T_{s_0})^{n-1}$ (with $n\geq 1$) form a $k$-basis of $e_\xi\He_{\PGL_2}$. A direct computation shows that they are sent to the following $k$-linearly independent matrices:
\begin{align*}
e_\xi(T_\omega T_{s_0})^n\mapsto \begin{pmatrix}
    0 & -X^{n+1}+X^{n-1}\\
    0 & X^n
\end{pmatrix}, &\quad e_\xi(T_{s_0}T_\omega)^n\mapsto \begin{pmatrix}
0 & 0\\
X^{n-1} & X^n
\end{pmatrix}\\
e_\xi T_{s_0}(T_\omega T_{s_0})^{n-1}\mapsto \begin{pmatrix}
0 & 0 \\
0 & -X^{n-1}
\end{pmatrix},\quad \text{and} \quad &e_\xi T_\omega (T_{s_0}T_\omega)^{n-1}\mapsto\begin{pmatrix}
X^{n}-(1-\delta_{n1})X^{n-2} & X^{n+1}-X^{n-1}\\
-X^{n-1}& -X^{n}
\end{pmatrix}
\end{align*}
for all $n\geq 1$, where $\delta_{n1}$ denotes the Kronecker delta. The map is thus injective. Finally, using that $T_{s_0}+1$ and $-T_{s_0}$ are sent to \(\begin{psmallmatrix}
1 & 0\\
0 & 0
\end{psmallmatrix}\) and \(\begin{psmallmatrix}
0 & 0\\
0 & 1
\end{psmallmatrix}\) respectively, one deduces from the above formulae that $Z(e_\gamma\He_{\PGL_2})$ once again identifies with the scalar matrices and is thus isomorphic to the polynomial algebra $k[X]$.
\end{proof}

From the above, the centre of $\He_{\PGL_2}$ is once again Gorenstein and its non-regular components have trivial homotopy category. In analogy with the $\GL_2$ situation, we let $\M_{\PGL_2,\gamma}=Z(e_\gamma\He_{\PGL_2})^2$ for any $\gamma$ (regular or not) with the natural left action of $e_\gamma\He_{\PGL_2}$ coming from the homomorphisms to $M_2(Z(e_\gamma\He_{\PGL_2}))$ constructed above and we put $\M_{\PGL_2}=\bigoplus_\gamma\M_{\PGL_2,\gamma}$. We now see that we may remove our condition on $p$ in \Cref{Delta_equiv} for the regular components and obtain an analogue of \Cref{Quillen_equiv_GL2}.

\begin{prop}\label{Quillen_equiv_PGL2} Let $\G=\PGL_2$ and $\gamma\in (T(\F_q)^\vee/W_0)_{\reg}$.
\begin{enumerate}
    \item The functor $\Res^{e_\gamma\He}_{e_\gamma \He_{x_0}}:\Mod(e_\gamma\He)\to\Mod(e_\gamma\He_{x_0})$ is a right Quillen equivalence.
    \item Assume that $p>2$. Then the functor
    $$
    M\mapsto \M_{\PGL_2}\otimes_{Z(\He)} M
    $$
    defines both a left and right Quillen equivalence $\Mod(Z(\He))\to \Mod(\He)$.
\end{enumerate}
\end{prop}

\begin{proof}
For (i), we may argue as in \Cref{Delta_equiv} using the resolution \eqref{OllSch_res_comp}. It once again suffices to show that $e_\gamma\He_C^\dagger$ has finite global dimension. But $e_\gamma\He_C^\dagger$ becomes identified via the isomorphism in \Cref{matrix_description_H_PGL_2}(ii) with the subalgebra of $M_2(A)$ generated by \(\begin{psmallmatrix}
1 & 0\\
0 & 0
\end{psmallmatrix}\), \(\begin{psmallmatrix}
0 & 0\\
0 & 1
\end{psmallmatrix}\) and \(\begin{psmallmatrix}
0 & 1\\
1 & 0
\end{psmallmatrix}\). The latter subalgebra is simply equal to $M_2(k)$, which is simple and thus of global dimension 0 as required.

For (ii), the assumption on $p$ ensures by \Cref{non_reg_is_reg} that the only non-zero factors of $\Ho(\He)$ in \eqref{decomp_Ho_H} are those corresponding to orbits of regular characters. Then the argument goes through identically as for $\GL_2$ in \Cref{Quillen_equiv_GL2}.
\end{proof}

\begin{rem}
Suppose that we have $p=2$. In light of the description of the centre in \Cref{matrix_description_H_PGL_2}, the functor
$$
M\mapsto \M_{\PGL_2}\otimes_{Z(\He_{\PGL_2})} M
$$
now only induces a fully faithful embedding $\Ho(Z(\He_{\PGL_2}))\to \Ho(\He_{\PGL_2})$ with image the product of the regular components. Furthermore, it follows from the proof of \Cref{Quillen_equiv_PGL2}(i) that the Gorenstein algebra $\He_{\PGL_2,C}^\dagger$ has non-zero homotopy category only because of its non-regular components. Indeed, when $\gamma$ is non-regular then the algebra $e_\gamma\He_{\PGL_2,C}^\dagger$ is isomorphic to $k[\Z/2\Z]$, which is of infinite global dimension when $p=2$.
\end{rem}

\subsubsection{The case $\G=\SL_2$} This case is different as the regular components of $\He_{\SL_2}$ are no longer matrix algebras over their centre. In order to describe our analogue of the spherical module, it will be more convenient to first work with the affine Hecke algebra $\He_{\GL_2,\aff}$. We fix a regular orbit $\gamma=\{\xi, \xi^{s_0}\}\in T_{\GL_2}(\F_q)^\vee/W_0$ and consider the algebra $e_\gamma\He_{\GL_2,\aff}$. Explicitly, it is given by the generators $T_0\coloneqq e_\gamma T_{s_0}$, $T_1\coloneqq e_\gamma T_{s_1}$, $e_1\coloneqq e_{\xi}$, $e_2\coloneqq e_{\xi^{s_0}}$ and relations $T_i^2=0$ ($i=0,1$), $e_je_l=\delta_{jl}$, $e_1+e_2=1$ and $e_j T_i=T_ie_{3-j}$ (for $j=1,2$ and $i=0,1$). We have a $k$-vector space decomposition $e_\gamma\He_{\GL_2,\aff}=e_1\He_{\GL_2,\aff}\oplus e_2\He_{\GL_2,\aff}$, and the set
\begin{equation}\label{basis_reg_block}
\{e_i(T_0T_1)^m, e_i(T_1T_0)^m, e_iT_0(T_1T_0)^{m}, e_iT_1(T_0T_1)^{m}\mid m\geq 0\}
\end{equation}
is a $k$-basis of $e_i\He_{\GL_2,\aff}$ for each $i=1,2$.

As a final piece of notation, we write $A_e$ for the $k$-span of $\{X_1^{2m}, X_2^{2m}\mid m\geq 0\}$ in $A$ and similarly $A_o$ for the $k$-span of $\{X_1^{2m+1}, X_2^{2m+1}\mid m\geq 0\}$ in $A$. Of course, $A_e$ is the $k$-subalgebra of $A$ generated by $X_1^2$ and $X_2^2$ and we have that $A=A_e\oplus A_o$ as $A_e$-modules.

\begin{prop}\label{matrix_descr_aff}
There is an injective $k$-algebra homomorphism $\rho:e_\gamma\He_{\GL_2,\aff}\to M_2(A)$ defined by
\[
\rho(e_1)=\begin{pmatrix}
1 & 0\\
0 & 0
\end{pmatrix}, \quad \rho(e_2)=\begin{pmatrix}
0 & 0\\
0 & 1
\end{pmatrix}, \quad \rho(T_0)=\begin{pmatrix}
0 & X_1\\
X_2 & 0
\end{pmatrix}\quad\text{and}\quad \quad \rho(T_1)=\begin{pmatrix}
0 & X_2\\
X_1 & 0
\end{pmatrix}\,.
\]
The map $\rho$ induces an isomorphism
\[
e_\gamma\He_{\GL_2,\aff}\cong \begin{pmatrix}
A_e & A_o\\
A_o & A_e
\end{pmatrix}
\]
of $k$-algebras, with the right hand side viewed as a subalgebra of $M_2(A)$, and also an isomorphism $Z(e_\gamma\He_{\GL_2,\aff})\cong A_e\cong A$.
\end{prop}

\begin{proof}
One checks immediately that the relations between the generators of $e_\gamma\He_{\GL_2,\aff}$ are preserved under the map $\rho$. Furthermore, we have for any $m\geq 0$ that
\begin{align*}
\rho((T_0T_1)^m)=\begin{pmatrix}
X_1^{2m} & 0\\
0 & X_2^{2m}
\end{pmatrix},& \quad \rho((T_1T_0)^m)=\begin{pmatrix}
X_2^{2m} & 0\\
0 & X_1^{2m}
\end{pmatrix}\\
\rho(T_0(T_1T_0)^{m})=\begin{pmatrix}
0 & X_1^{2m+1}\\
X_2^{2m+1} & 0
\end{pmatrix}\quad &\text{and}\quad \rho(T_1(T_0T_1)^{m})=\begin{pmatrix}
0 & X_2^{2m+1}\\
X_1^{2m+1} & 0
\end{pmatrix}\,.
\end{align*}
The remaining statements follow immediately from the above and \eqref{basis_reg_block}.
\end{proof}

To now link the above to $\He_{\SL_2}$, we first fix a generator $\zeta\in\F_q^\times$ and consider an arbitrary character $\xi\colon T_{\SL_2}(\F_q)\to k^\times$. Since $\F_q\subseteq k$, we may view $\xi$ as an $\F_q$-valued character of $T_{\SL_2}(\F_q)$ and one then has \(\xi\begin{psmallmatrix}
\zeta & 0\\
0 & \zeta^{-1}
\end{psmallmatrix}=\zeta^n\) for some $0\leq n\leq q-2$. Note then that $\xi$ admits $q-1$ distinct lifts to $T_{\GL_2}(\F_q)$ given by \(\tilde{\xi}_j\colon\begin{psmallmatrix}
\zeta^a & 0\\
0 & \zeta^{b}
\end{psmallmatrix}\mapsto\zeta^{aj-b(n-j)}\) (with $0\leq j\leq q-2$). If $\tilde{\xi}'$ is any character of $T_{\GL_2}(\F_q)$ with $\xi'\coloneqq \tilde{\xi}'|_{T_{\SL_2}(\F_q)}$, then it follows from the definition of the idempotent $e_\xi$ that $e_\xi e_{\tilde{\xi}'}=e_\xi e_{\xi'}$ and hence that
\begin{equation}\label{idemp_restr}
e_\xi=\sum_{j=0}^{q-2}e_{\tilde{\xi}_j}
\end{equation}
in $k[T_{\GL_2}(\F_q)]$. We fix one of the above lifts of $\xi$ and denote it by $\tilde{\xi}$ for simplicity. Let $\gamma\in T_{\SL_2}(\F_q)^\vee/W_0$ and $\tilde{\gamma}\in T_{\GL_2}(\F_q)^\vee/W_0$ denote the Weyl group orbits of $\xi$ and $\tilde\xi$ respectively. There is a $k$-algebra homomorphism $\psi_\gamma: e_\gamma\He_{\SL_2}\to e_{\tilde\gamma}\He_{\GL_2, \aff}$ obtained by composing the inclusion $e_\gamma\He_{\SL_2}\subseteq \He_{\GL_2, \aff}$ with the projection onto the $e_{\tilde\gamma}$-component, which by \eqref{idemp_restr} is injective and equal to multiplication by $e_{\tilde{\gamma}}$ in $\He_{\GL_2, \aff}$.

\begin{prop}\label{matrix_descr_SL2}
Keeping with the above notation, when $\gamma$ is regular then $\psi_\gamma$ is an isomorphism, and there is an injective $k$-algebra homomorphism $\rho:e_\gamma\He_{\SL_2}\to M_2(A)$ with image \(\begin{psmallmatrix}
A_e & A_o\\
A_o & A_e
\end{psmallmatrix}\), whose definition is identical to the one given in \Cref{matrix_descr_aff}, under which $Z(e_\gamma\He_{\SL_2})$ identifies with $A_e$-scalar matrices.
\end{prop}

\begin{proof}
Since the regular components of $\He_{\GL_2, \aff}=k[T_{\GL_2(\F_q)}]\otimes_{k[T_{\SL_2}(\F_q)]}\He_{\SL_2}$ have exactly the same description by generators and relations using the Iwahori-Matsumoto basis as those of $\He_{\SL_2}$, the map $\psi_\gamma$ is an isomorphism by construction. The rest follows from \Cref{matrix_descr_aff}.
\end{proof}

\begin{rem}\leavevmode
\begin{enumerate}
    \item The algebra $e_\gamma\He_{\SL_2}$ was described in \cite[\S 6.3.6]{AbThese} as an algebra of $2\times 2$-matrices with entries in certain morphism spaces with explicit generators and relations. The above may be viewed as a more concrete rewriting of that description.
    \item Explicitly, we obtain from the above that $Z(e_\gamma\He_{\SL_2})$ is generated by $e_1T_0T_1+e_2T_1T_0$ and $e_2T_0T_1+e_1T_1T_0$, which was also proved in \cite[Remark 3.5]{OS18}. The fact that we obtain an isomorphism to $A_e$ via the map $\rho$ may be thought of as keeping track of the fact that these two generators are built out of elements of the form $T_w$ with $\ell(w)=2$. By contrast, for $\GL_2$ the generators of the centre were built out of elements of the form $T_w$ with $\ell(w)=0$ or $1$.
\end{enumerate}
\end{rem}

We can now define the right analogue for us of the spherical module. Recall from \Cref{non-reg} that when $p>2$ there is a unique non-trivial and non-regular character $\sigma\in T_{\SL_2}(\F_q)$, which sends any fixed generator to $-1$. Continuing with our earlier notation, when $p>2$ and $\xi=\sigma$ we note that any choice of lift $\tilde \sigma$ to $T_{\GL_2}(\F_q)$ is now \emph{regular}, since non-regular characters of $T_{\GL_2}(\F_q)$ are always trivial on the $\SL_2$-torus (cf.\ \Cref{non-reg}). Thus $\tilde{\gamma}=\{\tilde{\sigma}, \tilde{\sigma}^{s_0}\}$ has size two even though $\gamma$ was a singleton. We may view again $e_{\tilde\gamma}\He_{\GL_2, \aff}$ as the $k$-algebra generated by $T_0$, $T_1$ and two orthogonal idempotents $e_1$, $e_2$ summing to $1$. The map $\psi_\sigma$ then identifies $e_\sigma\He_{\SL_2}$ with the $k$-subalgebra generated by $T_0$, $T_1$ (without the idempotents).

\begin{defn}
Let $\xi\in T_{\SL_2}(\F_q)^\vee$ be non-trivial and $\gamma$ denote its $W_0$-orbit. The restriction of the $e_{\tilde\gamma}\He_{\GL_2, \aff}$-action on $A^2$ via $\rho$ to $e_\gamma \He_{\SL_2}$ under the map $\psi_\gamma$ defines an $e_\gamma\He_{\SL_2}$-module which we denote by $\M_{\SL_2,\gamma}$ and call the \emph{affine spherical $e_\gamma\He_{\SL_2}$-module}. We will refer to the direct sum $\M_{\SL_2}=\bigoplus_{\gamma\neq \mathbf{1}}\M_{\SL_2,\gamma}$ as the \emph{affine spherical $\He_{\SL_2}$-module}.
\end{defn}

We note that the above definition does not depend up to isomorphism on the choice of lift $\tilde{\gamma}$. When $\gamma$ is regular, the affine spherical module is isomorphic to the one obtained by considering directly the action of $e_\gamma\He_{\SL_2}$ on $A^2$ via the direct map from \Cref{matrix_descr_SL2}. 

\begin{rem}
By definition of $\rho$, when specialising $\M_{\SL_2,\gamma}$ (for $\gamma$ regular) at the character of $A$ given by $X_1, X_2\mapsto 0$, this recovers the direct sum $\chi_1\oplus\chi_2$ of the two supersingular characters of $e_\gamma\He_{\SL_2}$ defined by $\chi_i(T_0)=\chi_i(T_1)=0$ and $\chi_i(e_j)=\delta_{ij}$. When specialising instead at the two characters $X_1\mapsto 0$, $X_2\mapsto \lambda$ and $X_1\mapsto \lambda$, $X_2\mapsto 0$ with $\lambda\neq 0$, this gives the two dimensional non-supersingular simple $e_\gamma\He_{\SL_2}$-modules denoted by $M_{12}(\lambda^2)$ and $M_{21}(\lambda^2)$ in \cite[Théorème 6.3.39]{AbThese}. By \textit{loc.\ cit}.\ we thus obtain all simple $e_\gamma\He_{\SL_2}$-modules in this way, analogously to the $\GL_2$ situation (cf.\ \cite[Theorem 7.4.9]{PeSch23}). In a similar way, one recovers the non-supersingular simple $e_\sigma\He_{\SL_2}$-modules (classified in \cite[Théorème 6.3.45]{AbThese}) by specialising $\M_{\SL_2,\sigma}$ at the closed points of $\Spec(A)$ outside the origin, and specialising at the origin gives a direct sum of two copies of the (unique) supersingular character in that component.
\end{rem}

Our affine spherical module is not just a $(\He_{\SL_2}, Z(\He_{\SL_2}))$-bimodule but in fact involves a larger algebra which we describe next. We will adopt the convention, for each $\mathbf{1}\neq\xi\in T_{\SL_2}(\F_q)^\vee$ with \(\xi\begin{psmallmatrix}
\zeta & 0\\
0 & \zeta^{-1}
\end{psmallmatrix}=\zeta^n\), to fix our chosen lift $\tilde{\xi}$ to be equal to $\tilde{\xi}_{\lceil \frac{n}{2}\rceil}$. With this in place and keeping with our earlier notation, we let
\begin{equation}\label{defn_zq}
\Zq\coloneqq Z(e_{\mathbf{1}}\He_{\SL_2})\times\prod_{\mathbf{1}\neq\gamma\in T_{\SL_2}(\F_q)^\vee/W_0} Z(e_{\tilde{\gamma}}\He_{\GL_2, \aff})\,.
\end{equation}
The ring $\Zq$ is the product of $Z(e_{\mathbf{1}}\He_{\SL_2})\cong k[X]$ (cf.\ \cite[Remark 3.5]{OS18}) with a product of copies of $A$ and hence is Gorenstein. There is an embedding of the centre $Z(\He_{\SL_2})$ into $\Zq$ induced by the product of the maps $\psi_\gamma$, which is an isomorphism on each component except for the $\sigma$-component, when $p>2$, where it identifies with the injection $k[X]\hookrightarrow A$ sending $X$ to $X_1+X_2$. By construction, the module $\M_{\SL_2}$ is actually an $(\He_{\SL_2}, \Zq)$-bimodule (where we let $Z(e_{\mathbf{1}}\He_{\SL_2})\subseteq \Zq$ act by zero).

\begin{lem}\label{adjun_SL_2}
The affine spherical module $\M_{\SL_2}$ is projective and induces an adjunction
\[
\Ho(\LM_{\M_{\SL_2}}): \Ho(\Zq)\rightleftarrows \Ho(\He_{\SL_2}):\Ho(\R_{\M_{\SL_2}})\,.
\]
\end{lem}

\begin{proof}
The second part follows from the first by \Cref{tensor_hom_lemma}. For the first part, we have to show that $\M_{\SL_2, \gamma}$ is projective over $e_\gamma\He_{\SL_2}$ for each $\gamma\neq \mathbf{1}$. To see this, we first observe that \(A^2\cong M_2(A)\cdot\begin{psmallmatrix}
1 & 0\\
0 & 0
\end{psmallmatrix}\) is projective over $M_2(A)$. Next, we observe that $M_2(A)$ is free as a left $e_{\tilde{\gamma}}\He_{\GL_2,\aff}$-module by \Cref{matrix_descr_aff}, since we have
\[
M_2(A)=\begin{pmatrix}
A_e & A_o\\
A_o & A_e
\end{pmatrix}\oplus \begin{pmatrix}
A_e & A_o\\
A_o & A_e
\end{pmatrix}\begin{pmatrix}
0 & 1\\
1 & 0
\end{pmatrix}\,.
\]
This proves our claim whenever $\gamma$ is regular by \Cref{matrix_descr_SL2}. For $p>2$ and the $\sigma$-component, the claim follows because $e_{\tilde{\gamma}}\He_{\GL_2,\aff}= \psi_\sigma(e_{\sigma}\He_{\SL_2})e_{\tilde{\sigma}}\oplus \psi_\sigma(e_{\sigma}\He_{\SL_2})e_{\tilde{\sigma}^{s_0}}$ and so is free over $e_{\sigma}\He_{\SL_2}$.
\end{proof}

Unlike the other two groups, the above adjunction is not an equivalence. We will describe it explicitly in the next subsection, see \Cref{Quillen_equiv_SL2}. We may however describe $\Ho(\He_{\SL_2})$ without it. Indeed, using that there is an natural isomorphism $\He_{x_0}\cong \He_{x_1}$ (for all three choices of $\G$) induced by conjugation by $\omega\in\GL_2(\Ff)$, we see from \Cref{Delta_equiv} that it will suffice to describe $\Ho(\He_{\SL_2,x_0})$ in order to determine $\Ho(\He_{\SL_2})$. For the regular components, this follows rather quickly from our earlier analysis.

\begin{lem}\label{Quillen_equiv_finite}
Let $\G\in\{\SL_2, \GL_2, \PGL_2\}$. For any regular $\gamma$ we have an equivalence
\[
\Ho(e_\gamma\He_{x_0})\simeq\Ho(A)\,.
\]
\end{lem}

\begin{proof}
As we will see in the next subection, the algebra $e_\gamma\He_{x_0}$ does not depend on $\G$ nor on the specific choice of regular $\gamma$, cf.\ the relations \eqref{rel_R} below. We therefore might as well assume that $\G=\PGL_2$. We then obtain the claimed equivalence by combining \Cref{Quillen_equiv_PGL2}(i), \Cref{Morita_equiv} and \Cref{matrix_description_H_PGL_2}(ii).
\end{proof}

From \Cref{Delta_equiv} and \Cref{Quillen_equiv_finite}, $\Ho(e_\gamma\He_{\SL_2})$ is equivalent to $\Ho(A)\times\Ho(A)$ while $\Ho(Z(e_\gamma\He_{\SL_2}))$ is equivalent to $\Ho(A)$ by \Cref{matrix_descr_SL2}. Furthermore, by \cite[Remark 3.5]{OS18} the centre $Z(e_\sigma\He_{\SL_2})$ is a polynomial algebra and hence has homotopy category equal to zero, while $\Ho(e_\sigma\He_{\SL_2})\neq 0$ (cf.\ \Cref{Ho_non_reg_case}). Therefore, the centre does not `see' the full homotopy category anymore. Nevertheless, we have that the regular part of $\Ho(\He_{\SL_2})$ is again a product of copies of $\Ho(A)$, similarly as for $\PGL_2$. In the next subsection, we will compute the non-regular component for $p>2$ as well as describe $\Ho(A)$ explicitly.

\subsection{The homotopy category of $\He_{x_0}$}\label{Section_R} In this section, we give an explicit description of the category $\Ho(\He_{x_0})$ for any of our three choices of $\G$ and prove \Cref{thmA}.

We first fix a regular orbit $\gamma$ and a representative $\xi$ of it. Then $e_\gamma\He_{x_0}$ identifies with the $k$-algebra $R$ generated by two primitive idempotents $e_1$ and $e_2$ and an element $T$, satisfying the relations
\begin{equation}\label{rel_R}
e_1+e_2=1, \quad e_ie_j=\delta_{ij}e_i,\quad e_iT=Te_{3-i}, \quad\text{and}\quad T^2=0
\end{equation}
for $1\leq i, j\leq 2$, by identifying $e_1$, $e_2$ and $T$ with $e_\xi$, $e_{\xi^{s_0}}$ and $e_\gamma T_{s_0}$ respectively, cf.\ the proof of \cite[Lemma 7.5]{Koz}. The ring $R$ has dimension 4 over $k$ with basis $e_1$, $e_2$, $Te_1$ and $Te_2$.

Next, there are exactly two $k$-valued characters $\chi_1$, $\chi_2$ of $R$ defined by
\begin{equation}\label{defn_chi}
\chi_i(T)=0\quad\text{and}\quad \chi_i(e_j)=\delta_{ij}
\end{equation}
for $1\leq i, j\leq 2$. Furthermore, we have a decomposition $R=Re_1\oplus Re_2$ of left $R$-modules, and for $i=1, 2$ we have a non-split short exact sequence
\begin{equation}\label{extn_chi_R}
0\to \chi_{3-i}\to Re_i\to \chi_i\to 0.
\end{equation}
The modules $\chi_i$ and $Re_i$ ($i=1,2$) are indecomposable. In fact, they are the only ones.

\begin{lem}\label{classification_indec}
Suppose that $M\in\Mod(R)$ is a non-zero indecomposable module. Then $M$ is isomorphic to one of $\chi_1, \chi_2$, $Re_1$ or $Re_2$.
\end{lem}

\begin{proof}
First, assume that there exists $m\in M$ such that $Tm\neq 0$. Then there is some $i=1,2$ such that $Te_im\neq 0$, and in fact $Te_im\notin k\cdot e_im$ since $T^2=0$. Thus the $R$-linear map $\varphi: Re_i\to M$, $re_i\mapsto re_im$, is injective. Since $R$ is self-injective, the classes of projective and injective modules coincide so that $Re_i$ is injective. It follows that $\varphi$ splits and thus must be an isomorphism by indecomposability of $M$.

We are therefore only left with the case where $TM=0$. In that case, $M=e_1M\oplus e_2M$ is an $R$-linear decomposition and $M=e_iM$ for some $i=1,2$ by indecomposability. Hence $M$ is a direct sum of copies of $\chi_i$ and so must be isomorphic to $\chi_i$, again by indecomposability.
\end{proof}

\begin{prop}\label{classification_all}
For any $M\in\Mod(R)$, there is an isomorphism
$$
M\cong \chi_1^{\oplus\mathcal{I}_1}\oplus \chi_2^{\oplus\mathcal{I}_2}\oplus Re_1^{\oplus\mathcal{J}_1}\oplus Re_2^{\oplus\mathcal{J}_2}
$$
for some indexing sets $\mathcal{I}_i$ and $\mathcal{J}_i$ ($i=1, 2$).
\end{prop}

\begin{proof}
By \Cref{classification_indec}, the result holds for any finite dimensional $R$-module. Let $\{M_i\}_{i\in \mathcal{K}}$ denote the collection of all indecomposable projective submodules of $M$ (i.e.\ isomorphic to one of $Re_1$ and $Re_2$). We then let
$$
X=\left\{\mathcal{J}\subseteq \mathcal{K}\;\middle|\;\sum_{j\in \mathcal{J}}M_j \text{ is direct}\right\},
$$
ordered by inclusion. By Zorn's lemma, there is a maximal element $\mathcal{J}\in X$. Let $N=\sum_{j\in \mathcal{J}} M_j$, which is by definition a projective submodule of $M$ of the form $Re_1^{\oplus\mathcal{J}_1}\oplus Re_2^{\oplus\mathcal{J}_2}$.

As $R$ is self-injective, the module $N$ is injective and thus there is a submodule $N'$ such that $M=N\oplus N'$. Further, by maximality of $\mathcal{J}$ there is no $i\in\mathcal{K}$ such that $M_i\subseteq N'$ and hence every indecomposable submodule of $N'$ is isomorphic to either $\chi_1$ or $\chi_2$ by \Cref{classification_indec}. Since $N'$ is the sum of its finite dimensional submodules, it therefore follows that $N'$ is semisimple and thus is isomorphic to a module of the form $\chi_1^{\oplus\mathcal{I}_1}\oplus \chi_2^{\oplus\mathcal{I}_2}$ as required.
\end{proof}

\begin{rem}
The above argument shows more generally that if $\mathscr{R}$ denotes any finite dimensional, self-injective $k$-algebra with the property that any indecomposable $\mathscr{R}$-module is either projective or simple, then any $\mathscr{R}$-module is a direct sum of indecomposables. Note that this decomposition is in fact unique up to auto-bijection of the indexing set for the decomposition by the Azumaya-Krull-Remak-Schmidt theorem, cf.\ \cite[Theorem 2.12]{FacBook}.
\end{rem}

Using the above, we can give a very elementary description of $\Ho(R)$ as a triangulated category. Below, we will call a triangle in a triangulated category \emph{trivial} if it is isomorphic to a shift of a triangle of the form $X\xrightarrow{\id}X\to 0\to \Sigma X$. We note that finite direct sums of trivial triangles are always distinguished, cf.\ \cite[Remark 1.2.2]{Neeman}. We will need two distinct but elementary triangulated structures on semisimple abelian categories, which we define now.

\begin{defn}\label{triangle_modk2} Let $\mathscr{A}$ be a semisimple abelian category.
\begin{enumerate}
    \item We will call the \emph{trivial triangulated structure on $\mathscr{A}$} the one whose shift functor is the identity and with distinguished triangles the finite direct sums of the trivial triangles. We will write $\mathscr{A}_\triv$ for $\mathscr{A}$ equipped with this triangulated structure.
    \item Suppose $\mathscr{A}=\mathscr{A}_1\times \mathscr{A}_2$ is a product of two semisimple abelian categories. We define $\Sigma:\mathscr{A}\to \mathscr{A}$ to be the autoequivalence given by $(X, Y)\mapsto (Y, X)$. Moreover, we define a distinguished triangle in $\mathscr{A}$ to be a sequence
    $$
    X\xrightarrow{f}Y\xrightarrow{g}Z\xrightarrow{h}\Sigma X
    $$
    which is exact at $Y$ and $Z$ and which satisfies $\im(h)=\Sigma \Ker(f)$. We will write $\mathscr{A}_{2,\triv}$ for $\mathscr{A}$ equipped with this triangulated structure.
\end{enumerate}
\end{defn}

We point out that any distinguished triangle as in $\mathscr{A}_{2,\triv}$ is also a finite direct sum of trivial triangles by semisimplicity of $\mathscr{A}$. It was shown in \cite[Section 8.1]{F14} that $\mathscr{A}_\triv$ is a well-defined triangulated structure on $\mathscr{A}$, and a completely analogous argument shows that $\mathscr{A}_{2,\triv}$ is also well-defined. In what follows, we will apply the construction of (ii) to $\Mod(k^2)$, identified with the product category $\Mod(k)\times\Mod(k)$. It was pointed out to us by Martin Kalck that the triangulated structure this induces on the compact objects may also be defined as the orbit category $D^b(\Mod_{\mathrm{f.g.}}(k))/[2]$ in the sense of \cite{Kel05}.

\begin{cor}\label{htpy_cat_Hx}
The functor $\Mod(k^2)_{2,\triv}\to \Ho(R)$ defined by
$$
(V_1, V_2)\mapsto (\chi_1\otimes_k V_1)\oplus (\chi_2\otimes_k V_2)
$$
is a triangle equivalence.
\end{cor}

\begin{proof}
By \Cref{classification_all}, the objects of $\Ho(R)$ are all of the form $\chi_1^{\oplus\mathcal{I}_1}\oplus \chi_2^{\oplus\mathcal{I}_2}$ up to isomorphism. Fix an arbitrary $j=1, 2$. Using \eqref{extn_chi_R}, we obtain a projective resolution
$$
\cdots \xrightarrow{\cdot T}Re_j\xrightarrow{\cdot T}Re_{3-j}\xrightarrow{\cdot T}Re_j\to\chi_j\to 0
$$
from which we compute
$$
\Ext_R^i(\chi_j, \chi_j)=\begin{cases}
    k & \text{if $i$ is even}\\
    0 & \text{otherwise}
\end{cases}\,,
$$
cf.\ also \cite[Example 7.2]{Koz}. This implies in particular that $\chi_j\neq 0$ in $\Ho(R)$. Moreover, the sequence \eqref{extn_chi_R} shows that the shift functor on $\Ho(R)$ swaps $\chi_1$ and $\chi_2$. By \eqref{Hom_Ho3} we may now deduce that
$$
[\chi_j, \chi_j]_R=k\quad\text{and}\quad [\chi_j, \chi_{3-j}]_R=0\,.
$$
From these observations, we obtain that the functor $(V_1, V_2)\mapsto (\chi_1\otimes_k V_1)\oplus (\chi_2\otimes_k V_2)$ is essentially surjective and fully faithful, i.e.\ an equivalence. By the above we also see that this functor commutes with the shifts. Since the finite direct sums of trivial triangles are always distinguished, our equivalence is therefore exact and we are done.
\end{proof}

By \Cref{Quillen_equiv_finite}, we now immediately deduce from the above that $\Ho(A)\simeq \Mod(k^2)$. We can make such an equivalence explicit as follows. Recall from \Cref{princ_example} that the $A$-modules $\M_1\coloneqq X_1A\cong A/X_2A$ and $\M_2\coloneqq X_2A\cong A/X_1A$ are Gorenstein projective of infinite projective dimension.

\begin{cor}\label{sing_cat_comp}
The functor $\Mod(k^2)_{2,\triv}\to\Ho(A)$ defined by
\[
(V_1,V_2)\longmapsto (\M_1\otimes_k V_1)\oplus (\M_2\otimes_k V_2)
\]
is a triangle equivalence.
\end{cor}

\begin{proof}
From \Cref{Quillen_equiv_finite} and \Cref{htpy_cat_Hx}, we deduce that there is a triangle equivalence $\Ho(A)\simeq \Mod(k^2)$. But the complete resolution and the Ext computation in \Cref{princ_example} together show that $\M_2\cong \Sigma\M_1$ and that
\[
[\M_i,\M_j]_A=\begin{cases}
k & \text{if $i=j$}\\
0 & \text{otherwise}
\end{cases}
\]
by \eqref{Hom_Ho3}. Our equivalence can therefore only identify $\M_1$ and $\M_2$ with the two generators of $\Mod(k^2)$, and up to postcomposing by $\Sigma$ it is therefore of the claimed form.
\end{proof}

\begin{rem}\label{Quillen_equiv1}
The equivalence of \Cref{sing_cat_comp} can also be constructed from a Quillen adjunction between $\Mod(A)$ and $\Mod(R)$. Indeed, under the isomorphism $e_\gamma\He_{\PGL_2}\cong M_2(A)$ from \Cref{matrix_description_H_PGL_2}, we may view $R$ as a subring of $M_2(A)$. Explicitly, this is given via \(e_1\mapsto \begin{psmallmatrix}
1 & 0\\
0 & 0
\end{psmallmatrix}\), \(e_2\mapsto \begin{psmallmatrix}
0 & 0\\
0 & 1
\end{psmallmatrix}\) and \(T\mapsto \begin{psmallmatrix}
0 & X_1\\
X_2 & 0
\end{psmallmatrix}\). Note that $M_2(A)$ is free over $R$ because $e_\gamma\He_{\PGL_2}$ is free over $e_\gamma\He_{\PGL_2, x_0}$. We consider $X=A^2$ as an $(A,R)$-bimodule, with right $R$-action given by the usual right action of matrices on row vectors. Using that $X\cong e_1M_2(A)\cong e_2M_2(A)$ as right $R$-modules, we deduce that $X$ is projective over $R$. We may thus apply \Cref{tensor_hom_lemma} and get a left Quillen functor $\LM_X:\Mod(R)\to\Mod(A)$. One checks directly that its right adjoint is given by the composite
$$
\Mod(A)\xrightarrow{A^2\otimes_{A}(-)}\Mod(M_2(A))\simeq\Mod(e_\gamma \He_{\PGL_2})\xrightarrow{\Res}\Mod(e_\gamma \He_{\PGL_2,x_0})\simeq\Mod(R)\,,
$$
where the first functor is the equivalence from \Cref{Morita_equiv}. \Cref{Quillen_equiv_PGL2}(i) then implies that $\LM_X$ is a left Quillen equivalence. A direct calculation (which we omit) shows that $\LM_X(\chi_i)=\M_i$ for $i=1,2$.
\end{rem}

We are left to investigate $\Ho(e_\sigma\He_{\SL_2, x_0})$. In that case, there is an isomorphism $e_\sigma\He_{\SL_2,x_0}\cong k[T]/(T^2)$. Note that $k[T]/(T^2)$ has a unique character $\chi$ sending $T$ to zero.

\begin{prop}\label{Ho_non_reg_case} \leavevmode
\begin{enumerate}
    \item For $M\in\Mod(k[T]/(T^2))$, there are indexing sets $\mathcal{I}$ and $\mathcal{J}$ such that $M\cong \chi^{\oplus\mathcal{I}}\oplus (k[T]/(T^2))^{\oplus\mathcal{J}}$.
    \item The functor $\Mod(k)_{\triv}\to\Ho(k[T]/(T^2))$, $V\mapsto \chi\otimes_k V$, is a triangle equivalence.
\end{enumerate}
\end{prop}

\begin{proof}
This is done by arguing completely analogously to \Cref{classification_all} and \Cref{htpy_cat_Hx}.
\end{proof}

We finally get an explicit description of $\Ho(\He)$ for $\G=\PGL_2$ or $\SL_2$.

\begin{thm}\label{description_PGL2_SL2}
There is an equivalence
\[
\Ho(\He_{\SL_2})\simeq \C\times \prod_{\text{$\gamma$ regular}} \left(\Mod(k^2)_{2,\triv}\times \Mod(k^2)_{2,\triv}\right),
\]
of triangulated categories, where $\C=0$ if $p=2$ and $\C=\Mod(k)_\triv\times \Mod(k)_\triv$ if $p>2$. Assuming $p>2$, there is also an equivalence
\[
\Ho(\He_{\PGL_2})\simeq \prod_{\text{$\gamma$ regular}} \Mod(k^2)_{2,\triv}
\]
of triangulated categories.
\end{thm}

\begin{proof}
This follows immediately by putting together \Cref{Delta_equiv}, \Cref{non-reg}, \eqref{decomp_Ho_Hx}, \Cref{htpy_cat_Hx} and \Cref{Ho_non_reg_case}.
\end{proof}

Recall the definition of the ring $\Zq$ from \eqref{defn_zq}. By \Cref{sing_cat_comp}, we have
\begin{equation}\label{HoZ}
\Ho(\Zq)\simeq \Mod(k^2)_{2,\triv}\times \prod_{\text{$\gamma$ regular}} \Mod(k^2)_{2,\triv}\,.
\end{equation}
Using the above, we can describe explicitly the adjunction from \Cref{adjun_SL_2}. For this, we need to make the equivalence of \Cref{description_PGL2_SL2} fully explicit in the regular case. Recall that we fixed a representative $\xi$ of each regular orbit $\gamma\in T_{\SL_2}(\F_q)$ and wrote $e_1=e_\xi$ and $e_2=e_{\xi^{s_0}}$ for the corresponding idempotent in $e_\gamma \He_{\SL_2}$. Working inside $\He_{\GL_2}$ we have an equality $e_\gamma\He_{\SL_2,x_1}=T_\omega (e_\gamma\He_{\SL_2,x_0})T_{\omega}^{-1}$ for $\gamma$ regular, and we write $\chi_i^\omega$ ($i=1,2$) for the characters of $e_\gamma\He_{\SL_2,x_1}$ given by $\chi_i^\omega(h)=\chi_i(T_{\omega}^{-1}hT_{\omega})$. Note that now $\chi_i^\omega(e_i)=0$. Then just as in \Cref{htpy_cat_Hx} we have an equivalence $\Mod(k^2)\to \Ho(e_\gamma\He_{\SL_2,x_1})$, $(V_1, V_2)\mapsto (\chi_1^\omega\otimes_k V_1)\oplus (\chi_2^\omega\otimes_k V_2)$. We fix the explicit equivalence
\[
\Ho(e_\gamma\He_{\SL_2})\simeq \Ho(e_\gamma\He_{\SL_2,x_0})\times \Ho(e_\gamma\He_{\SL_2,x_1})\simeq \Mod(k^2)_{2,\triv}\times\Mod(k^2)_{2,\triv}
\]
obtained from these.

\begin{lem}\label{Quillen_equiv_SL2}
Under \eqref{HoZ} and the above equivalences, the functor $\Ho(\R_{\M_{\SL_2}}):\Ho(\He_{\SL_2})\to \Ho(\Zq)$ has the following description:
\begin{enumerate}
    \item for $\gamma$ regular, it coincides with $\Mod(k^2)_{2,\triv}\times \Mod(k^2)_{2,\triv}\to \Mod(k^2)_{2,\triv}$, $(X,Y)\mapsto X\oplus Y$; and
    \item on the $\sigma$-component for $p>2$, it coincides with $\Mod(k)_{\triv}\times \Mod(k)_{\triv}\to \Mod(k^2)_{2, \triv}$, $(V,W)\mapsto (V\oplus W, V\oplus W)$.
\end{enumerate}
\end{lem}

\begin{proof}
We will compute the left adjoint $\Ho(\LM_{\M_{\SL_2}})$ as it is easier. We fix a regular component $\gamma$ and must compute the image of the $A$-modules $\M_1=A/X_2A$ and $\M_2=A/X_1A$ under the composite
\[
\Mod(A)\xrightarrow{\LM_{\M_{\SL_2,\gamma}}}\Mod(e_\gamma\He_{\SL_2})\xrightarrow{\Delta}\Mod(e_\gamma\He_{\SL_2, x_0})\times \Mod(e_\gamma\He_{\SL_2, x_1})\,.
\]
Let $M$ be any $A$-module. Then, by the explicit map in \Cref{matrix_descr_aff}, the action of $e_\gamma\He_{\SL_2}$ on $\LM_{\M_{\SL_2,\gamma}}(M)=M^{\oplus 2}$ is given as follows:
\begin{align*}
e_1\cdot (f,g)=(f,0)\,&,\quad e_2\cdot (f,g)=(0,g)\\
e_\gamma T_{s_0}\cdot (f,g)=(X_1g, X_2f)\quad&\text{and}\quad e_\gamma T_{s_1}\cdot (f,g)=(X_2g, X_1f)\,.
\end{align*}
We keep the notation $R=e_\gamma\He_{x_0}$ and write $R^\omega$ for $e_\gamma\He_{x_1}$. From the above and identifying $\M_1\cong k[X_1]$, we deduce the decomposition
\[
\Res^{e_\gamma\He_{\SL_2}}_R\LM_{\M_{\SL_2,\gamma}}(\M_1)=k\cdot(1,0)\oplus\bigoplus_{n\geq 0}\left(k\cdot(0,X_1^{n})\oplus k\cdot(X_1^{n+1},0)\right)\cong \chi_1\oplus \bigoplus_{n\geq 0} Re_2\,. 
\]
The latter module is isomorphic to $\chi$ in $\Ho(R)$. Similarly, one obtains in $\Ho(R^\omega)$ an isomorphism $\chi_1^\omega\cong \Res^{e_\gamma\He_{\SL_2}}_{R^\omega}\LM_{\M_{\SL_2,\gamma}}(\M_1)$ and, analogously, that
\[
\Delta(\LM_{\M_{\SL_2,\gamma}}(\M_2))\cong (\chi_2, \chi_2^\omega)
\]
in $\Ho(R)\times\Ho(R^\omega)$. Hence we have obtained under the equivalence of \Cref{sing_cat_comp} that $\Ho(\LM_{\M_{\SL_2,\gamma}})$ identifies with the diagonal functor $\Mod(k^2)\to\Mod(k^2)\times\Mod(k^2)$, $X\mapsto (X,X)$. One checks immediately that this is left adjoint to the functor in (i). The computation for (ii) is completely analogous.
\end{proof}

\section{The singularity category of semisimple Galois representations}\label{section_sing}

In this section, we give a Galois theoretic interpretation of the category $\Ho(\He)$. More precisely, we relate the homotopy categories $\Ho(\He)$ in semisimple rank one to the singularity categories of certain schemes parametrising the Galois representations occuring on the other side of the mod-$p$ Langlands program.

\subsection{Singularity categories} We first recall some elementary notions related to singularity categories. Following Krause \cite{Kra05}, our definition of singularity categories is as follows.

\begin{defn}\label{defn_sing_cat} Let $X$ be a separated Noetherian scheme over $k$. The \emph{singularity category} of $X$ is defined to be $\Sing(X)\coloneqq \mathbf{K}_\ac(\Inj(X))$, the category of acyclic complexes of injective quasi-coherent sheaves on $X$ up to chain homotopy.
\end{defn}

When $X=\Spec(S)$ is affine we will often also write $\Sing(S)$ for $\Sing(X)$.

\begin{rem}\label{geom_Lang}\leavevmode
\begin{enumerate}
    \item The singularity category of a Gorenstein ring $S$ was first investigated by Buchweitz \cite{Buch87}, who defined it to be the Verdier quotient $D^b(\mathrm{Coh}(S))/\mathrm{Perf(S)}$ of the bounded derived category of finitely generated $S$-modules by the thick subcategory of perfect complexes. This definition was taken up and later studied systematically for a large family of schemes by Orlov \cite{Orl04}. In general, one has that $\Sing(X)$ as we defined it is compactly generated and its compact objects are equivalent to the idempotent completion of $D^b(\mathrm{Coh}(X))/\mathrm{Perf(X)}$ in a natural way, cf.\ \cite[Theorem 1.1]{Kra05}. In particular, one deduces that $\Sing(X)=0$ if and only if $X$ is regular.
    \item Singularity categories also feature in the geometric Langlands program, see for example \cite[Appendix H]{ArGa15}. In that context, these categories are defined as the quotient $\Ind\Coh(X)/\QCoh(X)$ of the ind-completion of the (derived infinity) category of coherent sheaves on $X$ by the quasi-coherent sheaves, and $X$ itself is allowed to be a more general stack. The quotient category $\Ind\Coh(X)/\QCoh(X)$ is known to be compactly generated by $\Coh(X)/\mathrm{Perf}(X)$ in relatively high generality, cf.\ Remark H.1.5 of \textit{loc.\ cit}., and hence agrees with our definition for separated Noetherian schemes by the above. Finally, we note that there is a 6-functor formalism on $\Ind\Coh(X)$, see \cite[\S 8.7]{Sch6FF}, and we expect that the pushforward/pullback functors between singularity categories which we use later in this section can be recovered from that far more general construction. In practice however, we will only work with affine schemes or chains of projective lines in this paper, so that \Cref{defn_sing_cat} and the constructions we introduce later will suffice for our purposes.
\end{enumerate}
\end{rem}

Let $S$ be a commutative Noetherian ring. We recall that Becker showed that $\Sing(S)$ can be constructed as the homotopy category of the so-called \emph{singular injective} model structure on on the category $\Ch(S)$ of unbounded chain complexes of $S$-modules, cf.\ \cite[Proposition 2.2.1]{Beck14}. When $S$ is Gorenstein, the functor $\iota_0:\Mod(S)\to\Ch(S)$ sending an $S$-module to the corresponding chain complex concentrated in degree $0$ is a left Quillen equivalence for the Gorenstein injective model structure on $\Mod(S)$ and the singular injective model structure on $\Ch(S)$, cf.\ the dual of Proposition 3.1.3 in \textit{loc.\ cit}. There is also a projective analogue of this singular model structure with $\mathbf{K}_\ac(\Proj(S))$ as its homotopy category, and by \textit{loc.\ cit}.\ the functor $\iota_0$ is then a right Quillen equivalence for the Gorenstein projective model structure. Recalling that the Gorenstein projective and injective model structures have naturally the same homotopy categories (cf.\ \cite[Remark 2.12]{Dup25}), we therefore obtain the following.

\begin{prop}\label{GPmodel_sing}
Let $S$ be a commutative Gorenstein ring. Then $\Ho(S)$ is canonically equivalent to $\Sing(S)$.
\end{prop}

Identifying $\Ho(S)$ with $\mathrm{GInj}(S)/\Inj(S)$, the functor giving the aforementioned equivalence maps the homotopy class of an acyclic complex of injectives $\M_\bullet$ to the homotopy class of its zero-th cycles $Z_0(\M_\bullet))$.

\begin{rem}\label{support_rmk}\leavevmode
\begin{enumerate}
    \item Using \Cref{GPmodel_sing}, we can reinterpret \Cref{sing_cat_comp} as saying that the singularity category of two affine lines meeting at the origin is equivalent to $\Mod(k^2)$. This result is well-known to the experts, see e.g.\ \cite[Propositions 5.4.13-14]{Sym22} or \cite[Example 3.4]{Kal15}, however the arguments there are very different and typically rely on techniques from the representation theory of quivers. Our work therefore gives a completely new proof of the computation of this singularity category using Hecke algebras. Meanwhile, \Cref{Ho_non_reg_case} is a special case of the computation of the singularity categories of the infinitesimal thickenings of a point, which was done in \cite[\S 3.3]{Orl04}.
    \item We are now switching from a projective language (Gorenstein projective modules and acyclic complexes of projectives) to an injective language (Gorenstein injective modules and acyclic complexes of injectives). We point out that the passage between those is natural, since Gorenstein rings have a dualising complex and tensoring with it provides an equivalence between the homotopy categories of complexes of projectives and of injectives, cf.\ \cite[Theorem I]{IyKr06}.
\end{enumerate}
\end{rem}

Next, we recall that given a morphism $f:X\to Y$ of separated Noetherian schemes the pushforward functor $f_*$ induces a functor $\Sing(f_*):\Sing(X)\to \Sing(Y)$, cf.\ \cite[Theorem 1.5]{Kra05}. It is in fact shown in the proof of \textit{loc.\ cit}.\ that $f_*$ applied termwise sends acyclic complexes of injectives to acyclic complexes, and that $\Sing(f_*)$ has a right adjoint $\Sing(f^!)$. We will be primarily interested in the case where $f$ is flat, in which case $f^*$ is exact and $\Sing(f_*)$ is therefore computed by applying $f_*$ termwise to a complex. In that case, the exactness of $f^*$ also implies the existence of an induced functor $\Sing(f^*):\Sing(Y)\to\Sing(X)$ that is left adjoint to $\Sing(f_*)$, cf.\ Theorem 6.6 \& Remark 6.7 in \textit{loc.\ cit}. When $X=\Spec(A)$ and $Y=\Spec(B)$ are affine Gorenstein schemes, the adjoint pair of functors $\Sing(f^*)\vdash\Sing(f_*)$ then identifies under \Cref{GPmodel_sing} with the adjoint pair $\mathbf{L}\LM_A\vdash \mathbf{R}\R_A$ from \Cref{tensor_hom_lemma}.

In the special case of an open embedding $j:U\hookrightarrow X$, the isomorphism $j^*\circ j_*\cong \id$ implies that $\Sing(j_*)$ is fully faithful. Under a natural condition on the singularities, one even gets an equivalence as was originally shown by Orlov (cf.\ \cite[Proposition 1.14]{Orl04}).

\begin{prop}[{\cite[Theorem 1.6]{Kra05}}]\label{Orlov_equiv} Let $X$ be a separated Noetherian scheme of finite Krull dimension and suppose that $j:U\to X$ is a Zariski open embedding such that the open subscheme $U$ contains all the singular points of $X$. Then $\Sing(j_*)$ is an equivalence.
\end{prop}

Finally, we recall the notion of support in the singularity category $\Sing(X)$. Suppose that $\M_\bullet$ is a chain complexes of injective quasi-coherent sheaves on $X$. We define its \emph{support} to be
$$
\Supp(\M_\bullet)\coloneqq \bigcup_{n\in\Z} \Supp(\M_n)\,.
$$
This definition is not invariant under chain homotopy and thus does not naively descend to a notion of support on $\Sing(X)$. However, by \cite[Appendix B]{Kra05}, each homotopy class of complexes of injective quasi-coherent sheaves on $X$ is represented by a unique (up to chain isomorphism) so-called \emph{homotopically minimal} complex. Here, the homotopically minimal complexes are the complexes $\M_\bullet$ of injectives satisfying the property that any homotopy equivalence $\M_\bullet\to\M_\bullet$ is an isomorphism. We thus define the support of a chain homotopy class in $\Sing(X)$ to be the support of a homotopically minimal complex representing that class. We now describe how this notion translates under the equivalence of \Cref{GPmodel_sing}.

\begin{lem}\label{support_lem}
Let $R$ be a commutative Gorenstein ring and $M\in\Mod(R)$. Write $\underline{M}$ for the image of $M$ in $\Sing(R)$ under \Cref{GPmodel_sing}. Then
\begin{equation}\label{support_eq}
\Supp(\underline{M})=\left\{\mathfrak{p}\in\Spec(R)\mid \mathrm{injdim}_{R_\mathfrak{p}}(M_\mathfrak{p})=\infty\right\}\,.
\end{equation}
In particular, $\Supp(\underline{M})\subseteq \Supp(M)$ and we have equality when $M\cong Z_0(\M_\bullet)$ for some homotopically minimal acyclic complex of injectives $\M_\bullet$.
\end{lem}

\begin{proof}
We first recall that an acyclic complex of injectives $\mathcal{M}_\bullet$ is homotopically minimal if and only if the inclusion $Z_i(\mathcal{M}_\bullet)\subseteq \mathcal{M}_i$ of the $i$-th cycles is an injective envelope for each $i\in\Z$, cf.\ \cite[Lemma B.1]{Kra05}. Assuming that holds, we first claim that $\M_\bullet=0$ if and only if $M_i\coloneqq Z_i(\mathcal{M}_\bullet)$ is injective for some (or, equivalently, every) $i\in \Z$. Indeed, if $M_i$ is injective then it must equal its injective envelope $\M_i$. By exactness this forces $M_{i+1}=0$ and thus the injective envelope $\M_{i+1}$ is zero as well. Exactness now forces $M_{i+2}=0$ and iterating one obtains that $\M_j=0$ for all $j>i$. Meanwhile, since $M_{i+1}=0$ it follows that each $M_j$ with $j<i$ has finite injective dimension. Being Gorenstein injective, each such $M_j$ must therefore be injective and hence $M_j=\M_j$. It follows that every differential in $\M_\bullet$ is zero and thus $\M_\bullet=0$ by exactness, as required. In particular, we have show that $M_i=0$ if and only if $M_i$ is injective (or, equivalently, has finite injective dimension) for any $i\in\Z$. The final statement will therefore follow from \eqref{support_eq}.

To show \eqref{support_eq}, we fix $Q_f(M)$ a Gorenstein injective replacement of $M$, i.e.\ such that there is an injection $\iota: M\to Q_f(M)$ with $\mathrm{injdim}_R(\coker(\iota))<\infty$, and a homotopically minimal complex $\M_\bullet$ representing $\underline{M}$. We then have $M\cong Z_0(\M_\bullet)$ in $\Ho(R)$, which by \cite[Proposition 9.2]{Hov2} implies that there are injective $R$-modules $E$ and $E'$ such that $Q_f(M)\oplus E\cong Z_0(\M_\bullet)\oplus E'$. Fix $\mathfrak{p}\in\Spec(R)$ and recall that $R_\mathfrak{p}$ is Gorenstein (cf.\ \Cref{Laurent_Gorenstein}). We then have that $\mathrm{injdim}_{R_\mathfrak{p}}(M_\mathfrak{p})<\infty$ if and only if $Q_f(M)_\mathfrak{p}$ is injective, if and only if $Z_0(\M_\bullet)$ is injective. Since $(\mathcal{M}_\bullet)_\mathfrak{p}$ is a homotopically minimal acyclic complex of injective $R_\mathfrak{p}$-modules by \cite[Corollary 1.3]{Ba62}, these are equivalent to $(\mathcal{M}_\bullet)_\mathfrak{p}=0$ by the above and we are done.
\end{proof}

Now let $j:U\to X$ be once again an open embedding and let $Z=X\setminus U$. One then always has a functorial triangle
\begin{equation}\label{trianle_support}
\M_{Z,\bullet}\to \M_\bullet\to (j_*\circ j^*)\M_\bullet\xrightarrow{+1}
\end{equation}
where $\Supp(\M_{Z,\bullet})\subseteq Z$. In fact, let $\Sing_Z(X)\subseteq\Sing(X)$ denote the localising subcategory consisting of those (homotopically minimal) complexes whose support is contained in $Z$. One has alternatively that $\Sing_Z(X)=\Ker(\Sing(j^*))$, cf.\ \cite[Lemma 6.8]{Kra05}. The triangulated categories $\Sing(U)$, $\Sing(X)$ and $\Sing_Z(X)$ in fact fit together in a recollement as explained in Proposition 6.9 and its following paragraph in \textit{loc.\ cit}. Using the above, we now prove that the singularity category behaves well under nice Zariski covers.

\begin{prop}\label{sheaf_property} Let $X$ be a separated Noetherian scheme. Suppose that $(U_i)_{i=1}^n$ is an open cover such that $U_i\cap U_j$ has no singular points for each $i\neq j$. If $f:\bigsqcup_{i=1}^n U_i\twoheadrightarrow X$ is the corresponding covering morphism, then
$$
\Sing(f_*):\prod_{i=1}^n\Sing(U_i)\to \Sing(X)
$$
is an equivalence.
\end{prop}

\begin{proof}
By taking $U=\bigcup_{i=1}^{n-1}U_i$ and $V=U_n$ and arguing by induction, we may assume that $n=2$ so that $U=U_1$ and $V=U_2$. We also let $Z=X\setminus U$ and $Z'=X\setminus V$, and write $j:U\hookrightarrow X$ and $j':V\hookrightarrow X$ for the inclusions. By assumption $U\cap V$ has no singularities, hence $\Sing(U\cap V)=0$ and we deduce from \eqref{trianle_support} applied to $U\cap V\subseteq U$ that $\Sing(U)=\Sing_{Z'}(U)$. It follows that $\Sing(j_*)$ has essential image contained in $\Sing_{Z'}(X)$, and we claim that this is actually an equality. Indeed, let $\M_\bullet\in \Sing_{Z'}(X)$ and note that also $(j_*\circ j^*)(\M_\bullet)\in \Sing_{Z'}(X)$. The distinguished triangle \eqref{trianle_support} now implies that $\M_{Z,\bullet}$ is supported on $Z'$ as well. But since $U$ and $V$ cover $X$, we have $Z\cap Z'=\emptyset$ and so $\M_{Z,\bullet}=0$, which implies by \eqref{trianle_support} that $\M_\bullet$ lies in the essential image of $j_*$ as required.

The above implies that the essential images of $\Sing(j_*)$ and $\Sing(j'_*)$ are equal to $\Ker(\Sing(j'^*))$ and $\Ker(\Sing(j^*))$, respectively, cf.\ \cite[Lemma 6.8]{Kra05}. Hence it follows by adjointness that $\Sing(j_*)$ and $\Sing(j'_*)$ have orthogonal essential images in $\Sing(X)$, in the sense there are no non-zero morphisms betweem them, and thus $\Sing(f_*)=\Sing(j_*)\oplus\Sing(j'_*)$ is fully faithful. It remains to show that it is essentially surjective. For this, we let $\M_\bullet\in\Sing(X)$ be arbitrary and apply the unit of the adjunction $\Sing(f^*)\vdash\Sing(f_*)$ to \eqref{trianle_support} to get a morphism of triangles
\[
\begin{tikzcd}
\M_{Z,\bullet}\arrow{r}\arrow{d} & \M_\bullet \arrow{r}\arrow{d} & j_*\circ j^*(\M_\bullet)\arrow{r}{+1} \arrow{d}{=} & \phantom{a}\\
(j'_*\circ j'^*)(\M_\bullet)\arrow{r} & (j'_*\circ j'^*)(\M_\bullet)\oplus (j_*\circ j^*)(\M_\bullet)\arrow{r} & (j_*\circ j^*)(\M_\bullet)\arrow{r}{+1} & \phantom{a}
\end{tikzcd}
\]
where the outer vertical arrows are isomorphisms by the above. Thus so is the middle vertical arrow by \cite[014A]{stack} and we are done.
\end{proof}

\subsection{The scheme of semisimple Galois representations} We recall from \cite[\S 4]{PeSch25_2} the definition of a scheme $X_{q, \GL_2}$ whose geometric points parametrise semisimple 2-dimensional Galois representations. In this section, we assume that $p>2$ and following \textit{loc.\ cit}.\ we set $N_q\coloneqq \{0,\ldots, q-2\}$.

\subsubsection{The chain of projective lines} First, fix an arbitrary $l\in\Z_{\geq 1}$ and write $C_i\coloneqq \Pro^1_k$ for each $0\leq i\leq l-1$. We then set
$$
C(l)\coloneqq C_0\bigcup_{\infty\; 0}C_1\bigcup_{\infty\; 0}\cdots \bigcup C_{l-2}\bigcup_{\infty\; 0}C_{l-1},
$$
meaning the $k$-scheme obtained by taking a chain of $l$ copies of $\Pro^1_k$ and identifying the point at infinity of $C_{i}$ with the origin of $C_{i+1}$ for each $0\leq i\leq l-2$.

\begin{defn}[{\cite[\S 4.4.1]{PeSch25_2}}]\label{chain_P1} We let
$$
X_{n,q}\coloneqq \begin{cases}
C(\frac{q-1}{2})\times\Gm & \text{if $n$ is even}\\
C(\frac{q+1}{2})\times\Gm & \text{if $n$ is odd}
\end{cases}
$$
for $n\in N_q$. Then the scheme $X_{q, \GL_2}\coloneqq \bigsqcup_{n\in N_q}X_{n,q}$ is called the \emph{$q$-scheme of semisimple two dimensional mod-$p$ Galois representations}.
\end{defn}

As a justification for the terminology, it was shown in Theorem 4.5.1 of \textit{loc.\ cit}.\ that $X_{q, \GL_2}(k)$ is canonically (up to a sign choice) in bijection with the set of isomorphism classes of semisimple representations $\rho:\Gal(\overline{\Ff}/\Ff)\to\GL_2(k)$.

We recall our notation $A=k[X_1, X_2]/(X_1X_2)$. We write $\Af^1\cup_0\Af^1\coloneqq \Spec\left(A\right)$ for two affine $k$-lines meeting at the origin. We will refer to the closed subscheme $\Spec(k[X_j])\subseteq \Af^1\cup_0 \Af^1$ as the $X_j$-copy of $\Af^1$ in $\Af^1\cup_0 \Af^1$ (for $j=1,2$). Then, given any $l\geq 1$ and $0\leq i< l-1$, there is a natural open embedding
\begin{equation}\label{std_embed}
j_i:\Af^1\cup_0\Af^1\hookrightarrow C(l)
\end{equation}
which identifies $X_1$-copy of $\Af^1$ with $C_i$ minus the origin and the $X_2$-copy of $\Af^1$ with $C_{i+1}$ minus the point at infinity.

\subsubsection{The quotient morphism from Hecke to Galois} Next, we fix $\G=\GL_2$ and recall the construction of the morphism $L:\Spec(Z(\He))\to X_{q, \GL_2}$ from \cite[Theorem 6.2 \& Definition 7.3]{PeSch25_2}. Recall the product decomposition
$$
Z(\He)\cong \prod_{\text{$\gamma$ non-regular}} k[X, Z^{\pm1}]\times \prod_{\text{$\gamma$ regular}} k[X_1, X_2, Z^{\pm1}]/(X_1X_2)
$$
where the factors range over the $W_0$-conjugacy classes of $k$-valued characters of $T(\F_q)$ and $Z$ refers to the image of $T_{\omega^2}$ in each factor. We now relabel the factors a bit differently. For this, we again fix a generator $\zeta$ of $\F_q^\times$ and a representative $\xi$ for each $\gamma$. Then there is a unique $n\in N_q$ such that $\xi(\zeta\cdot\id)=\zeta^n$, independent of the choice of representative. For a given such $n$, we will write
\begin{equation}\label{partition_characters}
(T(\F_q)^\vee/W_0)_n=\{\gamma\in T(\F_q)^\vee/W_0\;|\; \text{$\xi(\zeta\cdot\id)=\zeta^n$ for all $\xi\in\gamma$}\}\,.
\end{equation}
We then rewrite the above decomposition as $Z(\He)=\prod_{n\in N_q} Z(\He)_n$, where
\begin{equation}\label{decomp_by_n}
Z(\He)_n\coloneqq \prod_{\gamma\in (T(\F_q)^\vee/W_0)_n} Z(e_\gamma\He)\;.
\end{equation}
As $q$ is odd, note that $n$ is even whenever $\xi$ is non-regular. For $n$ odd, we therefore have that
$$
\Spec(Z(\He)_n)\cong \left(\left(\Af^1\cup_0 \Af^1 \right)\bigsqcup\cdots\bigsqcup\left(\Af^1\cup_0 \Af^1 \right)\right)\times\Gm
$$
is the disjoint of $\frac{q-1}{2}$ copies of $(\Af^1\cup_0 \Af^1)\times\Gm$, the components being in correspondence with the (necessarily regular) elements of $(T(\F_q)^\vee/W_0)_n$. Meanwhile, for $n$ even, we have
$$
\Spec(Z(\He)_n)\cong \left(\Af^1\bigsqcup \left(\Af^1\cup_0 \Af^1 \right)\bigsqcup\cdots\bigsqcup\left(\Af^1\cup_0 \Af^1 \right)\bigsqcup \Af^1\right)\times\Gm,
$$
with $\frac{q-3}{2}$ middle terms corresponding to the regular elements and two outer terms corresponding to the non-regular elements of $(T(\F_q)^\vee/W_0)_n$. The morphism $L$ is constructed as the coproduct of certain morphisms $L_n:\Spec(Z(\He)_n)\to X_{n,q}$, which we describe below. In fact, we have $L_n=L_n'\times\id_{\Gm}$ so that we only describe $L'_n$, i.e.\ we describe $L_n$ on any given fiber of the natural projections to $\Gm$.

When $n$ is even, we may write $n=2s$ with $0\leq s\leq \frac{q-1}{2}$. For each $0\leq i\leq \frac{q-1}{2}$ there is a unique $\gamma_i=[\xi_i]\in (T(\F_q)^\vee/W_0)_n$ such that \(\xi_i\begin{psmallmatrix}
\zeta & 0\\
0 & 1
\end{psmallmatrix}=\zeta^{s+i}\) and we then put $U_i=\Spec(Z(e_{\gamma_i}\He))$. We then have $\Spec(Z(\He)_n)=\left(U_0\bigsqcup\cdots\bigsqcup U_{\frac{q-1}{2}}\right)\times\Gm$. On each connected component $U_i$ with $1\leq i\leq \frac{q-3}{2}$, we let $L'_n$ be the open embedding $j_{i-1}$ from \eqref{std_embed}. Finally, $L'_n$ on the two outer $\Af^1$'s (corresponding to $U_0$ and $U_{\frac{q-1}{2}}$) is the Zariski open embedding which identifies the leftmost, resp.\ rightmost, $\Af^1$ with $C_0$ minus the point at infinity, resp.\ $C_{\frac{q-3}{2}}$ minus the origin. Therefore we see that $L_n$ in the even case is in fact a Zariski cover.

When $n$ is odd, we may write $n=2s-1$ with $1\leq s\leq \frac{q-1}{2}$. For $1\leq i\leq \frac{q-1}{2}$ there is a unique $\gamma_i=[\xi_i]\in (T(\F_q)^\vee/W_0)_n$ such that \(\xi_i\begin{psmallmatrix}
\zeta & 0\\
0 & 1
\end{psmallmatrix}=-\zeta^{s+i-1}\) and we then put $U_i=\Spec(Z(e_{\gamma_i}\He))$, so that $\Spec(Z(\He)_n)= \left(U_1\bigsqcup\cdots\bigsqcup U_{\frac{q-1}{2}}\right)\times\Gm$. The map $L'_n$ on a component $U_i$ with $1<i<\frac{q-1}{2}$ is defined as $j_{i-1}$, completely analogously to the even case. When $i=1$, we identify the $X_2$-copy of $\Af^1$ in $U_1$ with $C_1$ minus the point at infinity. Meanwhile, for the $X_1$-copy of $\Af^1$ in $U_1$ we consider the map $\varphi:\Af^1\to \Pro^1=\Proj(k[Y_0, Y_1])$ defined by
\begin{equation}\label{defn_phi}
\varphi(t)\coloneqq \begin{cases}
\infty & \text{if $t=0$}\\
[t+t^{-1}:1] & \text{if $t\neq 0$}
\end{cases}
\end{equation}
and identify $\Pro^1$ with $C_0$ here. These two maps glue to a morphism $L'_n|_{U_1}$ since they agree at the origin. Writing $\tilde{Z}\coloneqq Y_0/Y_1$, we note that $\varphi$ really is algebraic since it is given by gluing
$$
D(X_1)=\Spec(k[X_1^{\pm 1}])\xrightarrow{\tilde{Z}\mapsto X_1+X_1^{-1}} \Spec (k[\tilde{Z}])\subseteq C_0
$$
with
$$
D(X_1^2+1)=\Spec(k[X_1, \frac{1}{X_1^2+1}])\xrightarrow{\tilde{Z}^{-1}\mapsto \frac{X_1}{X_1^2+1}}\Spec(k[\tilde{Z}^{-1}])\subseteq C_0.
$$
We completely analogously define $L'_n|_{U_{\frac{q-1}{2}}}$ to be given by gluing the identification of the $X_1$-line with $C_{\frac{q-3}{2}}$ minus the origin with the map $\varphi':\Af^1\to\Pro^1$ from the $X_2$-line to $C_{\frac{q-1}{2}}$, defined  as
\begin{equation}\label{defn_phi'}
\varphi'(t)\coloneqq \begin{cases}
[0:1] & \text{if $t=0$}\\
[1:t+t^{-1}] & \text{if $t\neq 0$}
\end{cases}
\end{equation}
This now concludes the definition of $L_n$ in the odd case and thus of $L$.

We recall that Gro{\ss}e-Kl{\"o}nne constructed a functorial bijection $M\mapsto V(M)$ between irreducible supersingular $\He$-modules and irreducible $2$-dimensional Galois representations over $k$, cf.\ \cite[Corollary 5.5]{GK20}. In the result below, we identify $X_{q, \GL_2}(k)$ with the set of semisimple 2-dimensional $k$-linear Galois representations under the explicit bijection given in \cite[Theorem 4.5.1]{PeSch25_2}. We also recall the equivalence $\Ho(\He)\simeq \Sing(Z(\He))$ from \Cref{Quillen_equiv_GL2}. Our main result is the following.

\begin{thm}\label{Langlands_singularities}
Let $\G=\GL_2$. Then the map $L:\Spec(Z(\He))\to X_{q, \GL_2}$ constructed above induces an equivalence
$$
\mathscr{L}_{\GL_2,*}:\Ho(\He)\simeq \Sing(Z(\He))\xrightarrow{\Sing(L_*)}\Sing (X_{q, \GL_2})
$$
of singularity categories. On objects, the map $M\mapsto \Supp(\mathscr{L}_{\GL_2,*}(M))$ gives rise to a well-defined injection
\[
\left\{\substack{\text{simple supersingular} \\ \text{$\He$-modules of} \\ \text{infinite projective dimension}}\right\}_{/\cong} \hookrightarrow \left\{ \substack{\text{irreducible}\\ \rho:\Gal(\overline{\Ff}/\Ff)\to \GL_2(k)}\right\}_{/\cong}
\]
which agrees with Gro{\ss}e-Kl{\"o}nne's map when $\Ff$ is a finite extension of $\mathbb{Q}_p$.
\end{thm}

We first prove a preparatory lemma. Recall our notation $B=A[Z^{\pm 1}]$.

\begin{lem}\label{localisation_lemma}
Let $f=X_2^2+1\in B$. Consider the injective map of $k$-algebras $\psi:B\to B_f\coloneqq B[\frac{1}{X_2^2+1}]$ defined by
$$
\psi(Z)=Z,\quad \psi(X_1)=X_1\quad\text{and}\quad \psi(X_2)=\frac{X_2}{X_2^2+1}.
$$
Letting $X\coloneqq B_f$ as a $(B_f,B)$-bimodule, where the $B$-action is the one coming from the map $\psi$, the adjunction
$$
\LM_X:\Mod(B)\rightleftarrows\Mod(B_f):\R_X
$$
is then a Quillen equivalence for the Gorenstein projective model structures.
\end{lem}

\begin{proof}
We have $B_f=A_f[Z^{\pm1}]$ and our aim will be to use \Cref{Laurent_extn} to prove the result. To that end, we first consider the map $\phi=\psi|_A: A\to A_f$. We claim that this map is flat. Indeed, the map $\Spec(k[X_2, \frac{1}{X_2^2+1}])\to \Spec(k[X_2])$ induced by $X_2\mapsto \frac{X_2}{X_2^2+1}$ is flat since $k[X_2]$ is a PID and thus flat modules coincide with torsion-free modules (cf.\ \cite[0AUW]{stack}). The map $\Spec(\phi)$ is obtained from it by base changing along the map $\Spec(A)\to \Spec(k[X_2])$ induced by the obvious inclusion of rings, hence is also flat (cf.\ \cite[01U9]{stack}).

We are therefore in the setting of \Cref{tensor_hom_lemma} and thus have a Quillen adjunction
$$
\phi^*=\LM_{A_f}:\Mod(A)\rightleftarrows \Mod(A_f):\R_{A_f}=\phi_*
$$
where both functors preserve weak equivalences, and we claim that it is a Quillen equivalence. For this, we first remark that the open embedding $\Spec(A_f)\subseteq \Spec(A)$ induces an equivalence of singularity categories by \Cref{Orlov_equiv}. In our language and by applying \Cref{sing_cat_comp}, this says that $\Ho(A_f)\simeq \Mod(k^2)$ with generators $A_f/(X_1)$ and $A_f/(X_2)$. But we have $\LM_{A_f}(A/(X_i))\cong A_f/(X_i)$ for $i=1$, $2$ and hence the functor $\mathbf{L}\LM_{A_f}$ is an equivalence by the explicit description of the two homotopy categories and since it commutes with arbitrary direct sums. By \Cref{Laurent_extn} we are now done.
\end{proof}

\begin{proof}[Proof of \Cref{Langlands_singularities}] We have $\Sing(L_*)=\prod_n \Sing((L_n)_*)$ so that we just need to show for the first part that $\Sing((L_n)_*):\Sing (Z(\He)_n)\to \Sing(X_{n,q})$ is an equivalence for each $n$. When $n$ is even, that follows immediately from \Cref{sheaf_property}. When $n$ is odd, we use the same notation as before and write $\Spec(Z(\He)_n)= \left(U_1\bigsqcup\cdots\bigsqcup U_{\frac{q-1}{2}}\right)\times\Gm$ with $U_i=\Af^1\cup_0\Af^1$ for each $i$. Viewing each $U_i\times\Gm$ as a Zariski open of $X_{n,q}$ via the embedding $j_{i-1}$ from \eqref{std_embed}, we see by combining \Cref{Orlov_equiv} and \Cref{sheaf_property} that
$$
\Sing(X_{n,q})\simeq\prod_{i=1}^{\frac{q-1}{2}}\Sing(U_i\times\Gm)\,.
$$
Under this equivalence, the functor $\Sing((L_n)_*)$ identifies with the endofunctor $\prod_iF_i$ of the product $\prod_{i}\Sing(U_i\times\Gm)$, defined by setting $F_i$ to be the identity for $1<i<\frac{q-1}{2}$ and, for $i=1$ or $\frac{q-1}{2}$, by setting $F_i$ to be the composite
$$
\Sing(U_i\times\Gm)\xrightarrow{\Sing(((L_n)|_{U_i\times\Gm})_*)}\Sing((\Pro^1\bigcup_{\infty\;0}\Pro^1)\times\Gm)\xrightarrow{\simeq}\Sing(U_i\times\Gm)
$$
where the second arrow is the equivalence from \Cref{sheaf_property}. We are left to show that $F_i$ is an equivalence when $i=1$ or $\frac{q-1}{2}$. For this, we consider the basic open affine
$$
V_i=\begin{cases}
D(X_2^2+1) & \text{if $i=1$}\\
D(X_1^2+1) & \text{if $i=\frac{q-1}{2}$}
\end{cases}
$$
of $U_i\times\Gm$ and note it suffices to show that $\Sing(V_i)\simeq \Sing(U_i\times\Gm)\xrightarrow{F_i}\Sing(U_i\times\Gm)$ is an equivalence, cf.\ \Cref{Orlov_equiv}. But by definition of $L_n$ on these components, this functor identifies with the equivalence $\Sing(\psi_*)$ from \Cref{localisation_lemma} for $i=1$ or its twist by swapping $X_1$ and $X_2$ for $i=\frac{q-1}{2}$, and so we are done with the proof that $\Sing(L_*)$ is an equivalence.

Finally, for the last part, we have that a supersingular simple $M$ of infinite projective dimension is sent to its central character under the equivalence of \Cref{Quillen_equiv_GL2}, and thus its image in $\Sing(Z(\He))$ is supported on the corresponding closed point by \Cref{support_lem}, and this point varies with $M$. As the morphisms $L_n$ induce bijections between the singular loci, it follows that $\mathscr{L}_*(M)$ is also supported on a single point in $X_{q, \GL_2}$ and so that the map $M\mapsto \Supp(\mathscr{L}_*(M))$ in the statement is well-defined and injective. That it agrees with Gro{\ss}e-Kl{\"o}nne's map for $\Ff$ an extension of $\mathbb{Q}_p$ is now a restatement of \cite[Theorem 8.9]{PeSch25_2}.
\end{proof}

\subsubsection{The case $\G=\PGL_2$} We may also use the map $L$ to give a Galois theoretic interpretation of $\Ho(\He_{\PGL_2})$. For this, we first recall that, under the assumption that $p>2$, the homotopy category $\Ho(\He_G)$ is by our earlier results equivalent to $\Ho(Z(\He_G))$ for $\G=\GL_2$ or $\PGL_2$. Furthermore, note that the surjection $\GL_2\to\PGL_2$ induces a natural bijection between $T_{\PGL_2}(\F_q)^\vee/W_0$ and the subset $(T_{\GL_2}(\F_q)^\vee/W_0)_0$ from \eqref{partition_characters}. From \eqref{matrix_description_H_PGL_2} we then have an isomorphism $Z(\He_{\PGL_2})\cong Z(\He_{\GL_2})_0/(T_{\omega^2}-1)$, cf.\ \eqref{decomp_by_n}. In what follows, we will consider the induced natural closed immersion
\[
i_\He:\Spec(Z(\He_{\PGL_2}))\hookrightarrow\Spec(Z(\He_{\GL_2}))
\]
and recall that it induces a pushforward functor $\Sing(i_{\He,*})$ between singularity categories. For simplicity of notation, we will also denote by $\Sing(i_{\He,*})$ the functor $\Ho(\He_{\PGL_2})\to\Ho(\He_{\GL_2})$ induced by the equivalences to the centres of \Cref{Quillen_equiv_GL2} and \Cref{Quillen_equiv_PGL2}.

Next, we note that by construction the simple supersingular $\He_{\PGL_2}$-modules of infinite projective dimension identify precisely, after inflating along the surjection $\He_{\GL_2}\to\He_{\PGL_2}$, with the simple supersingular $\He_{\GL_2}$-modules of infinite projective dimension occurring in a block $e_\gamma\He_{\GL_2}$ given by some $\gamma\in (T_{\GL_2}(\F_q)^\vee/W_0)_{0,\reg}$. We may therefore consider the Gro{\ss}e-Kl{\"o}nne functor $M\mapsto V(M)$ for the supersingular $\He_{\PGL_2}$-modules via this identification.

Finally, we consider the projection $\pr_2:X_{0,q}\to \mathbb{G}_m$ and let $X_{q, \PGL_2}\coloneqq\pr_2^{-1}(1)$. By definition this is a closed subscheme and we write $i_\Gal:X_{q, \PGL_2}\hookrightarrow X_{0,q}$ for the corresponding closed immersion. As a scheme, we see from \Cref{chain_P1} that $X_{q, \PGL_2}$ identifies with $C(\frac{q-1}{2})$, and note that by definition $L_0'$ defines a morphism $\Spec(Z(\He_{\PGL_2}))\to X_{q, \PGL_2}$. Further, under the bijection in \cite[Theorem 4.5.1]{PeSch25_2} the $k$-points of $X_{q, \PGL_2}$ identify with those semisimple 2-dimensional Galois representations with determinant given by Serre's fundamental character $\omega_f$, cf.\ \S 2.1 of \textit{loc.\ cit} for its definition.

\begin{prop}\label{Langlands_PGL2}
The morphism $L_0'$ induces an equivalence of categories
\[
\mathscr{L}_{\PGL_2,*}:\Ho(\He_{\PGL_2})\xrightarrow{\simeq}\Sing(Z(\He_{\PGL_2}))\xrightarrow{\Sing(L_0')_*} \Sing(X_{q, \PGL_2})\,.   
\]
On objects, the map $M\mapsto \Supp(\mathscr{L}_{\PGL_2,*}(M))$ gives rise to a well-defined injection
\[
\left\{\substack{\text{simple supersingular} \\ \text{$\He_{\PGL_2}$-modules of} \\ \text{infinite projective dimension}}\right\}_{/\cong} \hookrightarrow \left\{ \substack{\text{irreducible}\\ \rho:\Gal(\overline{\Ff}/\Ff)\to \GL_2(k)\\
\text{with $\det(\rho)= \omega_f$}}\right\}_{/\cong}
\]
which agrees with Gro{\ss}e-Kl{\"o}nne's map when $\Ff$ is a finite extension of $\mathbb{Q}_p$.
\end{prop}

\begin{proof}
It follows immediately by construction of $L_0'$ that \Cref{sheaf_property} applies and so that $\mathscr{L}_{\PGL_2,*}$ is an equivalence. From the definitions of all the maps, it follows that we have a commuting square
\[
\begin{tikzcd}
\Ho(\He_{\PGL_2})\arrow{r}{\mathscr{L}_{\PGL_2,*}} \arrow[swap]{d}{\Sing(i_{\He,*})} & \Sing(X_{q, \PGL_2}) \arrow{d}{i_{\Gal,*}}\\
\Ho(\He_{\GL_2}) \arrow{r}{\mathscr{L}_{\GL_2,*}} & \Sing(X_{0,q})
\end{tikzcd}
\]
and the injection in the statement follows from the one in \Cref{Langlands_singularities}.
\end{proof}

\begin{rem}
Using the bijection in \cite[Theorem 4.5.2]{PeSch25_2}, we can describe explicitly the irreducible Galois representations that do not lie in the images of our injections in \Cref{Langlands_singularities} and \Cref{Langlands_PGL2}. We use the same notation as in \textit{loc.\ cit}. For $\G=\PGL_2$, the missing Galois representations are $\ind(\omega_{2f})$ and $\ind(\omega_{2f})\otimes \omega_f^{\frac{q-1}{2}}$. For each \emph{even} $n\in N_q$ and $\lambda\in k^\times$, we next consider the character $\eta_{\lambda,n}\coloneqq \mathrm{unr}(\lambda)\otimes\omega_f^{\frac{n}{2}}$. Then the missing Galois representations in the $\GL_2$ case are more generally the $\ind(\omega_{2f})\otimes \eta_{\lambda,n}$ and $\ind(\omega_{2f})\otimes \omega_f^{\frac{q-1}{2}}\otimes \eta_{\lambda,n}$ for all $n\in N_q$ even and all $\lambda\in k^\times$ (we note that for fixed $n$ and $\lambda$ these correspond to points in $\pr_2^{-1}(\lambda^2)\subseteq X_{n,q}$).
\end{rem}

\subsubsection{The case $\G=\SL_2$} We can also describe the relationship between $\Ho(\He_{\SL_2})$ and a scheme whose connected components are chains of projective lines. First, we fix a generator $\zeta$ for $\F_q^\times$ and consider the characters $\xi_n\colon T_{\SL_2}(\F_q)\to k^\times$ defined by \(\xi_n\begin{psmallmatrix}
\zeta & 0\\
0 & \zeta^{-1}
\end{psmallmatrix}=\zeta^n\) with $n\in N_q$. We will call $\xi_n$ \emph{even}, resp.\ \emph{odd}, if $n$ is even, resp.\ odd, and we note that this property is preserved under the $W_0$-action since we assume that $p>2$.

Recall that there are exactly $q-1$ lifts $\tilde{\xi}_{0,n}, \ldots, \tilde{\xi}_{q-2,n}$ of this character to $T_{\GL_2}(\F_q)$, defined by \(\tilde{\xi}_{j,n}\begin{psmallmatrix}
\zeta^a & 0\\
0 & \zeta^b
\end{psmallmatrix}=\zeta^{aj-b(n-j)}\) for each $0\leq j\leq q-2$. If $\tilde{\gamma}_{j,n}$ denotes the $W_0$-orbit of $\tilde{\xi}_{j,n}$, note then that $\tilde{\gamma}_{j,n}\in (T_{\GL_2}(\F_q)^\vee/W_0)_{2j-n}$ (with the subscript interpreted modulo $q-1$), cf.\ \eqref{partition_characters}. Therefore, the parity of the action of these lifts on the centre of $\GL_2$ is well-defined and matches the parity of $n$. Recall from the definition of the ring $\Zq$, cf.\ \eqref{defn_zq}, that we fixed a lift $\tilde{\xi}_n\coloneqq \tilde{\xi}_{\lceil \frac{n}{2}\rceil,n}$ and we write $\tilde{\gamma}_n$ for the corresponding orbit. We then always have $\tilde{\gamma}_n\in (T_{\GL_2}(\F_q)^\vee/W_0)_{0}$ if $n$ is even and $\tilde{\gamma}_n\in (T_{\GL_2}(\F_q)^\vee/W_0)_1$ if $n$ is odd. We will thus consider the partition
$$
T_{\SL_2}(\F_q)^\vee/W_0=\left(T_{\SL_2}(\F_q)^\vee/W_0\right)_{\text{even}}\bigsqcup \left(T_{\SL_2}(\F_q)^\vee/W_0\right)_{\text{odd}}
$$
and, analogously to \eqref{decomp_by_n}, we consider by our convention the corresponding decomposition $\Zq=\Zq_{0}\times \Zq_{1}$. We remark that the sign $\sigma$ is even, resp.\ odd, if and only if $q\equiv 1\pmod{4}$, resp.\ $q\equiv 3\pmod{4}$. We next define our chains of projective lines for $\SL_2$.

\begin{defn}
For $m=0, 1$, we let
\[
Y_{m,q}\coloneqq\begin{cases}
C(\lfloor\frac{q+3}{4}\rfloor) & \text{if $m=0$}\\
C(\lceil\frac{q+3}{4}\rceil) & \text{if $m=1$}
\end{cases}
\]
and put $X_{q, \SL_2}\coloneqq Y_{0,q}\sqcup Y_{1,q}$.
\end{defn}

We now describe the construction of a morphism $L_{\SL_2}=L_0\sqcup L_1: \Spec(\Zq)\to X_{q, \SL_2}$. For this, we put $U_0=\Spec(Z(e_{\mathbf{1}}\He_{\SL_2}))$ and, for $1\leq i\leq \lfloor \frac{q-1}{4}\rfloor$ and $n=2i$, we let $U_i=\Spec(Z(e_{\tilde{\gamma}_n}\He_{\GL_2, \aff}))$. Then we have
\[
\Spec(\Zq_0)=\bigsqcup_{i=0}^{\lfloor \frac{q-1}{4}\rfloor} U_i\cong\Af^1\sqcup\bigsqcup_{i=1}^{\lfloor \frac{q-1}{4}\rfloor}(\Af^1\cup_0 \Af^1)\,.
\]
The morphism $L_0:\Spec(\Zq_0)\to Y_{0,q}$ is defined similarly to the $\GL_2$ case. Namely, we define $L_0|_{U_i}=j_{i-1}$ as defined in \eqref{std_embed} for $1\leq i< \lfloor \frac{q-1}{4}\rfloor$, we let $L_0|_{U_0}$ be the open embedding of $\Af^1$ as $C_0$ minus the point at infinity, and finally $L_0|_{U_{\lfloor\frac{q-1}{4}\rfloor}}$ is glued from the embedding of the $X_1$-copy of $\Af^1$ as $C_{\lfloor\frac{q-1}{4}\rfloor}$ minus the origin with the map $\phi'$ from \eqref{defn_phi'} on the $X_2$-copy of $\Af^1$.

Similarly, for $1\leq i\leq \lceil \frac{q-1}{4}\rceil$ and $n=2i-1$, we let $U_i=\Spec(Z(e_{\tilde{\gamma}_n}\He_{\GL_2, \aff}))$. Then we have
\[
\Spec(\Zq_1)=\bigsqcup_{i=1}^{\lceil \frac{q-1}{4}\rceil} U_i\cong\bigsqcup_{i=1}^{\lceil \frac{q-1}{4}\rceil}(\Af^1\cup_0\Af^1)\,.
\]
The morphism $L_1:\Spec(\Zq_1)\to Y_{1,q}$ is again defined similarly to the $\GL_2$ case. The map $L_1$ on a component $U_i$ with $1< i< \lceil \frac{q-1}{4}\rceil$ is again defined to be the map $j_{i-1}$. On $U_{\lceil \frac{q-1}{4}\rceil}$ the map $L_1$ is defined completely analogously to the even case using $\phi'$, and on $U_1$ it is defined by gluing the embedding of the $X_2$-copy of $\Af^1$ as $C_0$ minus the point at infinity with the map $\phi$ from \eqref{defn_phi} on the $X_2$-line.

Finally, recall that the character $\xi_n$ extends uniquely to a supersingular character $\chi_n$ of $\He_{\SL_2}$ whenever $n\neq 0$, and that the supersingular characters of $\He_{\SL_2}$ of infinite projective dimension are precisely $\chi_1, \ldots, \chi_{q-2}$ by \cite[Theorem 7.7]{Koz}. By arguing exactly as in the proof of \Cref{Langlands_singularities} and using the explicit functor $\Ho(\R_{\M_{\SL_2}}):\Ho(\He_{\SL_2})\to \Ho(\Zq)\simeq \Sing(\Zq)$ from \Cref{Quillen_equiv_SL2}, we obtain the following result.

\begin{prop}\label{Langlands_SL2}
The functor $\Sing(L)_*:\Sing(\Zq)\to \Sing(X_{q, \SL_2})$ is an equivalence. If we consider the induced functor
\[
\mathscr{L}_{\SL_2,*}:\Ho(\He_{\SL_2})\xrightarrow{\Ho(\R_{\M_{\SL_2}})}\Sing(\Zq)\xrightarrow{\Sing(L)_*}\Sing(X_{q, \SL_2})\,,
\]
then the map $\chi\mapsto \Supp(\mathscr{L}_{\SL_2,*}(\chi))$ defines a surjection
\[
\left\{\substack{\text{simple supersingular} \\ \text{$\He_{\SL_2}$-modules of} \\ \text{infinite projective dimension}}\right\}_{/\cong}\twoheadrightarrow \left\{\substack{\text{singular $k$-points}\\\text{of $X_{q, \SL_2}$}}\right\}
\]
with fibers $\{\chi_{\frac{q-1}{2}}\}$ and $\{\chi_{i}, \chi_{q-1-i}\}$ (for $0<i<\frac{q-1}{2}$).
\end{prop}

\begin{rem}\leavevmode
\begin{enumerate}
    \item The fibers of the map $\chi\mapsto \Supp(\mathscr{L}_{\SL_2,*}(\chi))$ are precisely the $L$-packets of those supersingular characters that have infinite projective dimension, as defined in \cite[Definition 6.4]{Koz2}. We note that it is shown in Corollary 7.14 of \textit{loc.\ cit}.\ that the collection of $L$-packets of size two we consider inject in the set of isomorphism classes of irreducible 2-dimensional projective Galois representations, in a way that is compatible with Gro{\ss}e-Kl{\"o}nne's functor.
    \item The $\sigma$-component of $\Zq$ maps down in $X_{q, \SL_2}$ to a $\Pro^1$ at the rightmost end of its chain (regardless of the parity of $\frac{q-1}{2}$). If we let $\tilde{X}_{q, \SL_2}$ denote the closed subscheme of $X_{q, \SL_2}$ obtained by removing that $\Pro^1$, then $\tilde{X}_{q, \SL_2}$ coincides with the scheme constructed by Ardakov-Schneider in \cite[Theorem 4.5.23]{ArdSch24}. The labeling of our components in $\Zq$ differs in the odd case to the one we used for $\GL_2$ but is instead analogous to the one used by Ardakov-Schneider in \textit{loc.\ cit}.
\end{enumerate}
\end{rem}

\section{DGAs and endomorphism rings}

For the rest of this paper, we assume that $\G=\GL_2$. We do not expect the homotopy category $\Ho(\He)$ to admit such a simple description as in \Cref{thmA} because the singular locus of $X_{q, \GL_2}$ is now 1-dimensional. In this section, we attempt to give some description of this category, by exhibiting an explicit generator for $\Ho(\He)$ which we use to describe this category as the derived category of an explicit DGA. We also compute the endomorphism rings in $\Ho(\He)$ of the simple supersingular modules.

\subsection{$\Ho(\He)$ as the derived category of a DGA}

It follows from very general considerations, cf.\ \cite[Theorem 8.6]{GilBook}, that the homotopy category of any Gorenstein ring is a so-called \emph{algebraic} triangulated category, i.e.\ triangle equivalent to the stable category of a Frobenius exact category, and it is also known to be compactly generated, cf.\ \cite[Theorem 9.4]{Hov2}. It follows from a classical theorem of Keller that any such triangulated category is triangle equivalent to the derived category of a DG-category, cf.\ \cite[Theorem 4.3]{Kel94}. More generally, when $X$ is a Noetherian separated scheme then $\Sing(X)$ is compactly generated and algebraic, cf.\ \cite[Theorem 1.1]{Kra05} and \cite[Lemma 7.5]{Kra07}, and thus is also equivalent to the derived category of a DG-category. A special case of this result of Keller is the following.

\begin{thm}[{\cite[Theorem 3.3]{Kel98}}]\label{Keller} Let $\C$ be a $k$-linear Frobenius exact category with arbitrary direct sums and let $\T$ be its stable category. Assume that $\T$ has a single compact generator $X$. Let $P_\bullet$ be an acyclic, termwise projective chain complex over $\C$ with $Z_0(P_\bullet)\cong X$, and set $\widetilde{\A}\coloneqq \underline{\End}(P_\bullet)$ to be the corresponding endomorphism DGA. Then there is an equivalence $\T\xrightarrow{\simeq}D(\widetilde{\A})$ of triangulated categories which sends $X$ to $\widetilde{\A}$.
\end{thm}

The above equivalence is constructed more or less explicitly. Indeed, if $\T'$ denotes the category of chain complexes of injective-projective objects of $\C$ up to chain homotopy, then the proof of \textit{loc.\ cit}.\ shows that $\underline{\Hom}(P_\bullet,-)$ defines an equivalence $\T'\simeq D(\widetilde{A})$. The theorem follows because taking zero-th cycles induces and equivalence between $\T'$ and $\T$.

We now explain how to make the above DGA explicit in the case of $\Ho(\He)$, or equivalently for the singularity category of the scheme of semisimple 2-dimensional Galois representations. By \Cref{Delta_equiv}, \Cref{non_reg_is_reg} and \Cref{non-reg}, it suffices to describe $\Ho(e_\gamma\He_{x_0}^\dagger)$ as the derived category of a DGA for $\gamma\in T(\F_q)^\vee/W_0$ regular. We keep our notation from \Cref{Section_R} and identify $e_\gamma \He_{x_0}$ with the ring $R$, so that $S\coloneqq R[Z^{\pm 1}]\cong e_\gamma \He_{x_0}^\dagger$. We recall that $R$ has two characters $\chi_1$, $\chi_2$ (cf.\ \eqref{defn_chi}). This induces two Gorenstein projective $S$-modules $\chi_1[Z^{\pm 1}]\coloneqq S\otimes_R \chi_1$ and $\chi_2[Z^{\pm 1}]\coloneqq S\otimes_R \chi_2$ by tensoring. 

\begin{prop}\label{generator_Ho} Suppose that $M\in \Mod(S)$ satisfies
$$
\left[\chi_1[Z^{\pm 1}], M\right]_S=0=\left[\chi_2[Z^{\pm 1}], M\right]_S\,.
$$
Then $M$ has image 0 in $\Ho(S)$, i.e.\ has finite projective dimension. In particular, the module $\chi_1[Z^{\pm 1}]\oplus \chi_2[Z^{\pm 1}]$ is a compact generator for the triangulated category $\Ho(S)$.
\end{prop}

\begin{proof} Since $\chi_1[Z^{\pm 1}]\oplus \chi_2[Z^{\pm 1}]$ is finitely generated Gorenstein projective, it must be compact in $\Ho(S)$ by \eqref{Hom_Ho1}. Next, we know from \Cref{htpy_cat_Hx} that $\chi_1\oplus\chi_2$ is a generator of $\Ho(R)$. Furthermore, we know from \Cref{trivial_res_lemma} that the two functors in the adjunction
\[
\Ho(\LM_S):\Ho(R)\rightleftarrows\Ho(S):\Ho(\R_S)
\]
are conservative, i.e.\ reflect isomorphisms. Since a left adjoint of a conservative functor always preserves a generator, the result follows as $\chi_1[Z^{\pm 1}]\oplus \chi_2[Z^{\pm 1}]=\LM_S(\chi_1\oplus\chi_2)$.
\end{proof}

Using the above, we can give a genuine $\He$-module which generates $\Ho(\He)$. In what follows, we write $\mathfrak{m}\coloneqq(X_1, X_2)\subseteq B$ for the maximal ideal. Then, for $\gamma$ regular we let $\M_\gamma^{\GP}\coloneqq\mathfrak{m}\M_\gamma$, where we identify $\M_\gamma$ with $\B^2$ via \Cref{matrix_description_H}, and call the direct sum $\M^{\GP}=\bigoplus_{\text{$\gamma$ regular}}\M_\gamma^{\GP}$ the \emph{Gorenstein projective spherical $\He$-module}. To see that this is really a Gorenstein projective module, we note that $\M_\gamma^\GP$ is the image of the Gorenstein projective module $\chi_1[Z^{\pm 1}]\oplus \chi_2[Z^{\pm 1}]\in \Mod(S)$ under the left Quillen equivalence
$$
\Mod(S)\xrightarrow{\mathcal{L}_{B^2}}\Mod(B)\xrightarrow{\simeq}\Mod(e_\gamma\He)
$$
for each $\gamma$, where the first arrow is the functor obtained from \Cref{Quillen_equiv1} after tensoring with $S$. This observation and \Cref{generator_Ho} immediately yield:

\begin{cor}
The category $\Ho(\He)$ is generated by $\M^{\GP}$.
\end{cor}

Going back to investigating a regular component $\Ho(e_\gamma\He_{x_0}^\dagger)$, we see by applying $S\otimes_R(-)$ to \eqref{extn_chi_R} that our generator $\chi_1[Z^{\pm 1}]\oplus \chi_2[Z^{\pm 1}]$ is the zero-th cycles of the complete resolution
\[
P_\bullet\coloneqq\cdots\xrightarrow{\cdot T} S \xrightarrow{\cdot T} S\xrightarrow{\cdot T}\cdots
\]
Consequently, we have by \Cref{Keller} that $\Ho(e_\gamma\He_{x_0}^\dagger)\simeq D(\E)$, where $\E\coloneqq\underline{\End}(P_\bullet)$. In order to describe this DGA explicitly, we first note that the complex $P_\bullet$ splits into a direct sum $P_{1,\bullet}\oplus P_{2, \bullet}$, where
\[
P_{1,\bullet}=\cdots\xrightarrow{\cdot T}Se_2\xrightarrow{\cdot T}Se_1\xrightarrow{\cdot T}Se_2\xrightarrow{\cdot T}\cdots 
\]
with $P_{1,0}=Se_1$ and $P_{2,\bullet}=P_{1,\bullet}[1]$. Writing $\E_{ij}\coloneqq\underline{\Hom}(P_{i,\bullet}, P_{j,\bullet})$, we thus have
\[
\E=\begin{pmatrix}
\E_{11} & \E_{12}\\
\E_{21} & \E_{22}
\end{pmatrix}\,.
\]
By the periodicity of $P_{i,\bullet}$, we actually have $\E_{11}=\E_{22}=\E_{12}[1]=\E_{21}[1]$ as complexes. In order to describe these, we consider the maps $Se_i\xrightarrow{\cdot T} Se_{3-i}$ for $i=1,2$ and by slight abuse of notation denote them both by $\tau$. We then have $\Hom_S(Se_i, Se_i)=k[Z^{\pm 1}]\cdot \id_{Se_i}$ and $\Hom_S(Se_i, Se_{3-i})=k[Z^{\pm 1}]\cdot \tau$ for $i=1, 2$. We may thus describe $\E_{11}$ as follows. We consider
\[
E_0\coloneqq\prod_{\Z}k[Z^{\pm1}]\quad\text{and} \quad E_1\coloneqq \prod_{\Z} k[Z^{\pm1}]\cdot \tau\,,
\]
viewed as subspaces of the algebra $\prod_{\Z} \left(k[\tau]/(\tau^2)\right)[Z^{\pm1}]$. By the definition of the internal Hom of chain complexes and using the above description of the morphisms between the terms of $P_{1,\bullet}$, we see that we have
\[
\E_{11}\cong \underline{E}\coloneqq \cdots\to E_1\to E_0\to E_1\to E_0\to \cdots\,,
\]
with even degree terms equal to $E_0$, and differential given by
\begin{equation}\label{differential}
d_n:(x_l)_{l\in \Z}\mapsto ((x_l-(-1)^nx_{l+1})\tau)_{l\in \Z}\,.
\end{equation}
In particular, we note that $d_n=0$ when $n$ is odd. We gather some facts about the DGA $\E$ below.

\begin{prop}\label{computation_DGA}
We have
\begin{equation}\label{descr_DGA}
\E \cong \begin{pmatrix}
\underline{E} & \underline{E}[1]\\
\underline{E}[1] & \underline{E}
\end{pmatrix}
\end{equation}
with $\underline{E}$ as described above. Furthermore, we have:
\begin{enumerate}
    \item there is a $k$-algebra isomorphism
    \[
    \E_0\cong \prod_{\Z} e_\gamma \He_{x_0}^\dagger\,\text{; and}
    \]
    \item if $H_\bullet$ denotes the cohomology algebra of $\E$ and $\mathscr{R}\coloneqq k[X^{\pm 1}, Z^{\pm 1}]$, then
    \[
    H_\bullet\cong \begin{pmatrix}
        \mathscr{R}_e & \mathscr{R}_o\\
        \mathscr{R}_o & \mathscr{R}_e
    \end{pmatrix}\subseteq M_2(\mathscr{R}),
    \]
    where $\mathscr{R}_e$, resp.\ $\mathscr{R}_o$, is the $k[Z^{\pm 1}]$-submodule of $\mathscr{R}$ generated by the even powers of $X$, resp.\ the odd powers of $X$, and where $Z^{\pm 1}$ has degree 0 and $X^n$ has degree $n$ (in any matrix position).
\end{enumerate}
\end{prop}

\begin{proof}
We just need to show (i) and (ii). For (i), we have
\[
\E_0 \cong \begin{pmatrix}
E_0 & E_1\\
E_1 & E_0
\end{pmatrix}\cong \prod_\Z \begin{pmatrix}
k[Z]^{\pm{1}} & k[Z]^{\pm{1}}\cdot \tau\\
k[Z]^{\pm{1}}\cdot \tau & k[Z]^{\pm{1}}
\end{pmatrix}\,.
\]
The isomorphism is then obtained by sending in each factor \(e_\gamma T_{\omega^2}\mapsto\begin{psmallmatrix}
Z & 0\\
0 & Z
\end{psmallmatrix}\), \(e_1\mapsto\begin{psmallmatrix}
1 & 0\\
0 & 0
\end{psmallmatrix}\), \(e_2\mapsto\begin{psmallmatrix}
0 & 0\\
0 & 1
\end{psmallmatrix}\) and \(e_\gamma T_{s_0}\mapsto\begin{psmallmatrix}
0 & \tau\\
\tau & 0
\end{psmallmatrix}\). For (ii), we deduce from \eqref{differential} that in $\underline{E}$ one has
\[
\Ker(d_n)=\begin{cases}
E_1 & \text{if $n$ is odd}\\
k[Z^{\pm 1}]\cdot (\ldots, 1,1,1,\ldots) & \text{if $n$ is even}
\end{cases}
\]
and
\[
\im(d_n)=\begin{cases}
0 & \text{if $n$ is odd}\\
E_1 & \text{if $n$ is even}
\end{cases}\,.
\]
Therefore the cohomology of $\underline{E}$ is zero in odd terms and isomorphic to $k[Z^{\pm 1}]$ in the even terms. We write $\iota_n$ for the cohomology class of $(\ldots, 1,1,1,\ldots)$ in $H_n(\underline{E})$ if $n$ is even, resp.\ in $H_n(\underline{E}[1])$ if $n$ is odd. From \eqref{descr_DGA} one deduces that
\[
H_\bullet=\left(\bigoplus_{\text{$n$ even}}\begin{pmatrix}
k[Z^{\pm 1}]\cdot \iota_n & 0\\
0 & k[Z^{\pm 1}]\cdot \iota_n
\end{pmatrix}\right)\oplus\left(\bigoplus_{\text{$n$ odd}}\begin{pmatrix}
0 & k[Z^{\pm 1}]\cdot \iota_n\\
k[Z^{\pm 1}]\cdot \iota_n & 0
\end{pmatrix}\right)
\]
is of the claimed form, by identifying $X^n$ with $\iota_n$.
\end{proof}


\subsection{Endomorphism rings of supersingular modules} We conclude by computing the endomorphisms in $\Ho(\He)$ of the simple supersingular modules of infinite projective dimension. We fix a regular orbit $\gamma$ and consider the rings $R$ and $S$ as in the previous section. For $i=1,2$ and each $\lambda\in k^\times$, we will write $\chi_{i,\lambda}\coloneqq \chi_i[Z^{\pm 1}]/(Z-\lambda)$, i.e.\ $\chi_{i,\lambda}$ is the one-dimensional $S$-module whose restriction to $R$ is $\chi_i$ and on which $Z$ acts by multiplication by $\lambda$. Recall from \Cref{recollections_Hecke} the definition of the simple supersingular $\He$-modules. There is a unique such module factoring through $e_\gamma\He$ and such that $T_{\omega}^2$ acts as $\lambda\cdot \id$, which we denote by $M_{\gamma, \lambda}$. It is associated to a supersingular character $\chi$ of $\Hea$ such that $\Res^{\Hea}_{T(\F_q)}(\chi)\in\gamma$, and in particular we have that $\Res^{e_\gamma\He}_{e_\gamma\He_{x_0}^\dagger}(M_{\gamma, \lambda})\cong \chi_{1,\lambda}\oplus \chi_{2,\lambda}$. Our aim is to show the following.

\begin{thm}\label{endom_comp}
There is a $k$-algebra isomorphism $\End_{\Ho(\He)}(M_{\gamma, \lambda})\cong e_\gamma\He_{x_0}$.
\end{thm}

By the equivalence from \Cref{Delta_equiv} we have
\[
\End_{\Ho(\He)}(M_{\gamma, \lambda})\cong\End_{\Ho(S)}(\chi_{1,\lambda}\oplus \chi_{2,\lambda})\cong\begin{pmatrix}
[\chi_{1,\lambda},\chi_{1,\lambda}]_S & [\chi_{1,\lambda},\chi_{2,\lambda}]_S\\
[\chi_{2,\lambda},\chi_{1,\lambda}]_S & [\chi_{2,\lambda},\chi_{2,\lambda}]_S
\end{pmatrix}\,,
\]
where we recall our notation $[-,-]_S\coloneqq \Hom_{\Ho(S)}(-,-)$. We first need to calculate each entry above. Fix $i\in\{1,2\}$. By definition, we have a short exact sequence
\begin{equation}\label{triangle}
0\to \chi_{i}[Z^{\pm 1}]\xrightarrow{(Z-\lambda)\cdot}\chi_{i}[Z^{\pm 1}]\xrightarrow{p_i} \chi_{i,\lambda}\to 0
\end{equation}
and it will be useful to first compute self-extensions of $\chi_{i}[Z^{\pm 1}]$. For this, we note that by tensoring \eqref{extn_chi_R} with $S$ we get a non-split extension
\begin{equation}\label{ext}
0\to\chi_{3-i}[Z^{\pm 1}]\to Se_i\to \chi_{i}[Z^{\pm 1}]\to 0\,.
\end{equation}
By arguing exactly as in the proof of \Cref{htpy_cat_Hx}, we obtain
\begin{equation}\label{ext_calc}
\Ext_S^j\left(\chi_{i}[Z^{\pm 1}],\chi_{i}[Z^{\pm 1}]\right)\cong\begin{cases}
k[Z^{\pm 1}] & \text{if $j$ is even}\\
0 & \text{otherwise}
\end{cases}\,.
\end{equation}

\begin{rem}
The projectivity of $Se_i$ implies by \eqref{ext} that $\chi_{i}[Z^{\pm 1}]\cong \Sigma\chi_{3-i}[Z^{\pm 1}]$ in $\Ho(S)$, and thus the above computation is equivalent to
\[
\left[\chi_{i}[Z^{\pm 1}], \chi_{j}[Z^{\pm 1}]\right]_S=\begin{cases}
k[Z^{\pm 1}] & \text{if $i=j$}\\
0 & \text{otherwise}
\end{cases}\,.
\]
In particular, the generator $\chi_1[Z^{\pm1}]\oplus \chi_2[Z^{\pm1}]$ has endomorphism ring $k[Z^{\pm 1}]^2$ in $\Ho(S)$.
\end{rem}

Using the above, we may now prove the following.

\begin{lem}\label{final_hom_dim}
We have
$$
\dim_k\left[\chi_{i,\lambda}, \chi_{j,\lambda}\right]_{S}=1
$$
for all $1\leq i, j\leq 2$ and any $\lambda\in k^\times$.
\end{lem}

\begin{proof}
We keep having $i\in\{1,2\}$ fixed. First note that $\chi_{i,\lambda}\cong \Sigma\chi_{3-i,\lambda}$ in $\Ho(S)$. Indeed, the short exact sequence \eqref{triangle} induces a distinguished triangle and we have a commuting diagram
\[
\begin{tikzcd}
 \chi_{3-i}[Z^{\pm 1}] \arrow{d}{\cong} \arrow{r}{(Z-\lambda)\cdot} & \chi_{3-i}[Z^{\pm 1}] \arrow{d}{\cong}\arrow{r}{} & \chi_{3-i,\lambda}\arrow{r}{+1} & \phantom{,}\\
 \Sigma \chi_i[Z^{\pm 1}] \arrow{r}{(Z-\lambda)\cdot} & \Sigma\chi_i[Z^{\pm 1}] \arrow{r}{} & \Sigma\chi_{i, \lambda}\arrow{r}{+1} &.
\end{tikzcd}
\]
By the axioms of triangulated categories and the triangulated 5-lemma (cf.\ \cite[014A]{stack}), the above can be extended to an isomorphism of triangles, which in particular gives us our claim. Using this, we deduce from \eqref{Hom_Ho3} that
$$
[\chi_i[Z^{\pm 1}], \chi_{3-i,\lambda}]_S\cong \Ext_S^1(\chi_i[Z^{\pm 1}], \chi_{i,\lambda})\quad\text{and}\quad [\chi_i[Z^{\pm 1}], \chi_{i,\lambda}]_S\cong \Ext_S^2(\chi_i[Z^{\pm 1}], \chi_{i,\lambda})
$$
for $i=1,2$. But we can compute this. Indeed, by considering the long exact sequence on $\Ext$ from \eqref{triangle} and by \eqref{ext_calc} we get an exact sequence
$$
0\to \Ext^1_S(\chi_i[Z^{\pm 1}],\chi_{i,\lambda} )\to k[Z^{\pm 1}]\xrightarrow{(Z-\lambda)\cdot} k[Z^{\pm 1}]\to \Ext^2_S(\chi_i[Z^{\pm 1}],\chi_{i,\lambda})\to 0.
$$
We deduce immediately that $[\chi_i[Z^{\pm 1}], \chi_{i,\lambda}]_S=k$ and $[\chi_i[Z^{\pm 1}], \chi_{3-i,\lambda}]_S=0$. From this we can now compute our Hom spaces by viewing \eqref{triangle} as a distinguished triangle in $\Ho(S)$ and applying the cohomological functor $[-,\chi_{3-i,\lambda}]_S$. Indeed, the associated long exact sequence now gives
$$
0\to [\chi_{3-i,\lambda},\chi_{3-i,\lambda}]_S\to [\chi_{3-i}[Z^{\pm 1}],\chi_{3-i,\lambda}]_S\xrightarrow{(-)\circ (Z-\lambda)\cdot } [\chi_{3-i}[Z^{\pm 1}],\chi_{3-i,\lambda}]_S\to [\chi_{i,\lambda}, \chi_{3-i,\lambda}]_S\to 0\,.
$$
The final observation we need to make is that the middle arrow $[\chi_{3-i}[Z^{\pm 1}],\chi_{3-i,\lambda}]_S\to [\chi_{3-i}[Z^{\pm 1}],\chi_{3-i,\lambda}]_S$ is the zero map. Indeed, both sides are just $S$-linear maps modulo homotopy (cf.\ \eqref{Hom_Ho1}) and any $S$-linear $\chi_{3-i}[Z^{\pm 1}]\to\chi_{3-i,\lambda}$ kills $(Z-\lambda)$ by definition of the character $\chi_{3-i,\lambda}$.  Thus we conclude from the above that there are isomorphisms
$$
[\chi_{3-i,\lambda},\chi_{3-i,\lambda}]_S\cong [\chi_{3-i}[Z^{\pm 1}],\chi_{3-i,\lambda}]_S\cong k \quad\text{and}\quad k\cong [\chi_{3-i}[Z^{\pm 1}],\chi_{3-i,\lambda}]_S\cong [\chi_{i,\lambda}, \chi_{3-i,\lambda}]_S
$$
and we are done.
\end{proof}

\begin{proof}[Proof of \Cref{endom_comp}]
We keep the notation from the proof of \Cref{final_hom_dim}. We saw there that the map $p_{3-i}$ from \eqref{triangle} induces an isomorphism $[\chi_{3-i, \lambda}, \chi_{3-i, \lambda}]_S\cong [\chi_{3-i}[Z^{\pm 1}], \chi_{3-i, \lambda}]_S$. It follows that $p_{3-i}$ is non-zero in $\Ho(S)$ and spans the one dimensional space $[\chi_{3-i}[Z^{\pm 1}], \chi_{3-i, \lambda}]_S$. Let $\partial: \Sigma^{-1}\chi_{3-i, \lambda}\to \chi_{3-i}[Z^{\pm 1}]$ be the connecting homomorphism in the triangle induced by \eqref{triangle}. We also saw that $\partial$ induces an isomorphism $[\chi_{3-i}[Z^{\pm 1}],\chi_{3-i, \lambda}]_S\cong [\chi_{i, \lambda}, \chi_{3-i, \lambda}]_S$, and thus any map in $[\chi_{i, \lambda}, \chi_{3-i, \lambda}]_S$ is a scalar multiple of the composite
\[
g_i:\chi_{i, \lambda}\cong \Sigma^{-1}\chi_{3-i, \lambda}\xrightarrow{\partial} \chi_{3-i, \lambda}[Z^{\pm 1}]\xrightarrow{p_{3-i}}\chi_{3-i, \lambda}\,.
\]
Since we have on the other hand $[\chi_{3-i}[Z^{\pm 1}],\chi_{i,\lambda}]_S=0$, it follows from the above that the composition map
$$
[\chi_{i,\lambda}, \chi_{3-i,\lambda}]_S\times [\chi_{3-i,\lambda}, \chi_{i,\lambda}]_S\to [\chi_{i,\lambda}, \chi_{i,\lambda}]_S
$$
is zero.

From \Cref{final_hom_dim}, we have that $\End_{\Ho(\He)}(M_{\gamma,\lambda})\cong \End_{\Ho(S)}(M_{\gamma,\lambda})$ is a $4$-dimensional $k$-algebra, with $k$-basis $\tilde{e}_1, \tilde{e}_2, \tilde{t}_1, \tilde{t}_2$ defined by
$$
\tilde{e}_i:\chi_{1,\lambda}\oplus\chi_{2,\lambda}\xrightarrow{\text{pr}_i}\chi_{i,\lambda}\xrightarrow{\text{incl}_i} \chi_{1,\lambda}\oplus\chi_{2,\lambda}
$$
and
$$
\tilde{t}_i:M=\chi_{1,\lambda}\oplus\chi_{2,\lambda}\xrightarrow{\text{pr}_i}\chi_{i,\lambda}\xrightarrow{g_i}\chi_{3-i, \lambda}\xrightarrow{\text{incl}_{3-i}}\chi_{1,\lambda}\oplus\chi_{2,\lambda}\,.
$$
Furthermore, by the above we see that these satisfy the relations
$$
\tilde{e}_1+\tilde{e}_2=\id, \quad \tilde{e}_i^2=\tilde{e}_i, \quad \tilde{e}_{3-i}\tilde{t}_i=\tilde{t}_i=\tilde{t}_i\tilde{e}_i,
$$
and
$$
\tilde{t}_i^2=\tilde{t}_i\tilde{t}_{3-i}=\tilde{e}_i\tilde{t}_i=\tilde{t}_i\tilde{e}_{3-i}=\tilde{e}_i\tilde{e}_{3-i}=0,
$$
for $i=1,2$. Thus the map $e_\gamma\He_{x_0}\cong R\to \End_{\Ho(\He)}(M_{\gamma,\lambda})$ defined by $e_i\mapsto \tilde{e}_i$ and $T\mapsto \tilde{t}_1+\tilde{t}_2$ is a well-defined algebra isomorphism.
\end{proof}

\bibliographystyle{abbrv}
\bibliography{homotopyendo}

\end{document}